\RequirePackage{fix-cm}
\documentclass[smallextended]{svjour3}       % onecolumn (second format)
\smartqed  % flush right qed marks, e.g. at end of proof
\usepackage{graphicx}

\usepackage{multirow}
\usepackage{latexsym,enumerate}
\usepackage{subfigure}
\usepackage{graphicx,epstopdf}
\usepackage{mathrsfs}
\usepackage{amsfonts}
\usepackage{amssymb,amsmath}
\usepackage{epsfig}       %for Postscript files
\usepackage{subfigure}    %for parallel figures
\usepackage{float}
\usepackage{array}
\usepackage{color}
\usepackage{listings}

\usepackage{mathptmx}      % use Times fonts if available on your TeX system
\usepackage[ruled]{algorithm2e}

\usepackage{diagbox}

\numberwithin{lemma}{section}
\numberwithin{theorem}{section}
\numberwithin{corollary}{section}
\numberwithin{definition}{section}
\numberwithin{proposition}{section}
\numberwithin{remark}{section}
\numberwithin{example}{section}
\numberwithin{figure}{section}
\begin{document}

\title{Accelerated forward-backward method with fast convergence rate for nonsmooth convex optimization beyond differentiability
	%Insert your title here%\thanks{Grants or other notes
}

\titlerunning{Accelerated forward-backward method for nonsmooth convex optimization}        % if too long for running head

\author{Wei Bian$^{*}$         \and
        Fan Wu %etc.
}


\institute{
	Wei Bian \at
	School of Mathematics, Harbin Institute of Technology, Harbin {\rm 150001}, P.R. China\\
	Institute of Advanced Study in Mathematics, Harbin Institute of Technology, Harbin {\rm 150001}, P.R. China\\
             % first address \\
             % Tel.: +123-45-678910\\
             % Fax: +123-45-678910\\
              \email{bianweilvse520@163.com}           %  \\
           \and
          Fan Wu \at
          School of Mathematics, Harbin Institute of Technology, Harbin {\rm 150001}, P.R. China\\
             % second address
             \email{wufanmath@163.com}
}

\date{Received: date / Accepted: date}

\maketitle

\begin{abstract}
We propose an accelerated forward-backward method with fast convergence rate for finding a minimizer of a decomposable nonsmooth convex function over a closed convex set, and name it smoothing accelerated proximal gradient (SAPG) algorithm. The proposed algorithm combines the smoothing method with the proximal gradient algorithm with extrapolation $\frac{k-1}{k+\alpha-1}$ and $\alpha>3$. The updating rule of smoothing parameter $\mu_k$ is a smart scheme and guarantees the global convergence rate of $o(\ln^{\sigma}k/k)$ with $\sigma\in(\frac{1}{2},1]$ on the objective function values. Moreover, we prove that the sequence is convergent to an optimal solution of the problem. Furthermore, we introduce an error term in the SAPG algorithm to get the inexact smoothing accelerated proximal gradient algorithm. And we obtain the same convergence results as the SAPG algorithm under the summability condition on the errors. Finally, numerical experiments show the effectiveness and efficiency of the proposed algorithm.


\end{abstract}

\section{Introduction}
\label{intro}
A popular optimization model that encompasses various convex problems arising in scientific and engineering
applications is the well-known composite minimization problem:
\begin{equation}\label{ob}
\min_{x\in\mathcal{X}}\quad f(x):=c(x)+g(x),
\end{equation}
where $\mathcal{X}$ is a nonempty closed convex subset of $\mathbb{R}^n$, $c:\mathbb{R}^n\rightarrow\mathbb{R}$ is a
continuous convex function on $\mathcal{X}$, and $g:\mathbb{R}^n\rightarrow\mathbb{R}\cup\{+\infty\}$ is a proper,
lower semicontinuous convex function. For a given $\vartheta>0$, the proximal mapping of $\vartheta g$
on $\mathcal{X}$, denoted by ${\rm prox}_{\vartheta g}$, is defined by
\begin{equation}\label{prox}
{\rm prox}_{\vartheta g}(y):=\arg\min_{x\in\mathcal{X}}\left\{\vartheta g(x)+\frac{1}{2}\|x-y\|^2\right\}.
\end{equation}
The efficient computation of such a proximal mapping is indispensable to a number of functions \cite{Adly2020,Beck2018,Bian2020,Lu2014}. Throughout
this paper, we focus on the case that the proximal operator of $g$ on $\mathcal{X}$ can be calculated effectively, and assume optimal solution set $\arg\min_{\mathcal{X}}f$ of (\ref{ob})
is nonempty.

In recent decades, first-order method is the leading method for solving large-scale optimization problems in real-world
applications, such as compressed sensing \cite{Candes1,Donoho}, image sciences \cite{Beck-Teboule,Chambolle},
variable selection \cite{Beck2018-MP}, etc.
Due to the decomposition structure of the objective function in (\ref{ob}), a mother scheme to solve it is the
forward-backward splitting method \cite{Bruck,Passty}, which is often called the proximal gradient (PG) method when
it is used to solve the convex programming. The main computational efforts of PG method in each iteration are the evaluations
of the gradient of the smooth part and the proximal calculation of the other part in the objective function.
Fukushima and Mine \cite{Fukushima} gave an earliest work on the analysis of PG method, and recently it has been widely
studied in many literatures. However, the original PG method is often slow, which is with the convergence rate of $O({1}/{k})$
on the objective function values. Then, various accelerated methods have been attempted to PG method, such as the fast iterative
shrinkage-thresholding algorithm (FISTA) proposed by Beck and Teboulle \cite{Beck-Teboule}, which extends the seminal work of
accelerated gradient algorithm for solving a class of smooth convex minimization introduced by Nesterov in 1983 \cite{Nesterov1983}.
 When $c$ is continuously differentiable, a typical accelerated proximal gradient (APG) method for (\ref{ob}) is to perform an extrapolation on the current iteration and
takes the following general form
\begin{equation}\label{Attouch}
\left\{
\begin{aligned}
& y^k=x^k+\beta_k(x^k-x^{k-1}) \\
& x^{k+1}= {\rm prox}_{\vartheta g}(y^k-\vartheta\nabla c(y^k)),
\end{aligned}\right.
\end{equation}
where $\vartheta$ is a positive constant depending on the Lipschitz constant of $\nabla c$, $\{\beta_k\}\subseteq[0,1]$ are the extrapolation
coefficients.
By setting $\beta_k=\frac{t_k-1}{t_{k+1}}$ with $t_1=1$ and $t_{k+1}=\frac{\sqrt{t_k^2+1}+1}{2}$,
FISTA exhibits a faster convergence rate of $O({1}/{k^2}) $ on the objective function values.
Based on the Nesterov's extrapolation scheme, the accelerated process in (\ref{Attouch})
has become increasingly important and be proven to be particularly useful in the first-order methods for solving the structured convex
minimization problems. Most recently, after some simplification,
Chambolle and Dossal \cite{Chambolle-Dossal} first proved the sequence convergence of FISTA, while it was independently settled in \cite{Attouch-MP}
for (\ref{Attouch}) with $\beta_k=\frac{k-1}{k+\alpha-1}$ and $\alpha>3$. What's more, Attouch and Peypouquet \cite{Attouch-Peypouquet} improved the
convergence rate of (\ref{Attouch}) on the function values from $O(1/k^2)$ to $o(1/k^2)$ by setting $\alpha>3$ instead of $\alpha=3$. Besides these
results, we refer the readers to \cite{Attouch2020C,Attouch2020F,Attouch2019,Attouch2020Convergence} and the references therein for more complementary and interesting results on the accelerated algorithms
based on the Nesterov's extrapolation techniques.

Despite we can loose the smoothness of $g$ in (\ref{ob}) due to its computability of proximal operator, a crucial and standard assumption common
to all the above mentioned PG and APG algorithms is the Lispchitz continuity of $\nabla c$. One may think that we can use the alternating direction
of multipliers (ADM) scheme \cite{Boyd2011,Hong2017,Yang2016} when both $c$ and $g$ are nonsmooth convex function, but have the computable proximal operator. Although ADM
scheme is effective for some composite models, there are several serious difficulties in doing this as discussed in \cite{Bauschke}. Recently, Bauschke,
Bolte and Teboulle \cite{Bauschke} derived an appropriate descent lemma on replacing the upper quadratic approximation of the smooth function by a
proximity measure with the Bregman distance. With this lemma, \cite{Bauschke} proved that the PG algorithm owned the $O({1}/{k})$ convergence rate on the objective
values when $c$ is convex and continuously differentiable (not necessarily with a global Lipschitz gradient). With some additional assumptions on the
Bregman proximal distance, the sequence convergence is also established in \cite{Bauschke}.
Almost at the same time, Nguyen \cite{Nguyen} independently analyzed the convergence of the PG method based on the Bregman distance for composite
minimization problems in general reflexive Banach spaces.
However, there are many problems in applications that can not be expressed by (\ref{ob}) with the sum of a continuously differentiable function satisfying the
conditions in \cite{Bauschke} and a function with computable proximal operator as the objective function. Some popular and interesting examples are listed in Section \ref{section2}.

For (\ref{ob}) with a general nonsmooth convex objective function, there are very few literatures. Based on the method in \cite{Nesterov1983} and an
appropriate smooth $\epsilon$-approximation of the initial nonsmooth objective function, Nesterov \cite{Nesterov2005} improved the efficiency estimate
of the order from $O\left({1}/{\epsilon^2}\right)$ to $O({{1}/{\epsilon}})$ for finding an $\epsilon$-solution $x^{\epsilon}$, i.e.
\[
f(x^{\epsilon})-\min f\leq\epsilon.
\]
And then this work was extended to the saddle-point problem arising in finding a Nash equilibria for games with the same order of efficiency estimate \cite{Hoda}.
Soon after, Chen \cite{Chen-MP} studied a class of smoothing methods for solving the constrained nonsmooth nonconvex optimization problem.
Based on the smoothing method in \cite{Chen-MP}, Zhang and Chen in \cite{zhang2009} proposed a smoothing projected gradient method for minimizing a nonsmooth nonconvex problem on a closed convex feasible set, and showed that any accumulation point generated by the method is a stationary point of the problem associated with a smoothing function. Recently, Bian and Chen \cite{Bian-Chen-SINMU} came up with a smoothing proximal gradient algorithm for the constrained $\ell_0$ penalized nonsmooth convex regression problem. In particular, the authors in \cite{Bian-Chen-SINMU} established that the local convergence rate of $o(1/k^{\tau})$ with any $\tau\in(0,{1}/{2})$ on the objective function values and the iterates converges to a local minimizer of the considered problem.
Similarly, inspired by the effect of smoothing method, Zhang and Chen \cite{zhang2020} proposed a smoothing active set method for linearly constrained non-Lipschitz nonconvex optimization and proved the local convergence of the method.
It is worth noting that Bian in \cite{Bian2020} independently developed a smoothing fast iterative shrinkage-thresholding algorithm based on the extrapolation coefficients of FISTA for solving problem \eqref{ob} and proved that the global convergence rate of the objective function values is $O(\ln k/k)$. Without the sequential convergence, only the optimality of the accumulated points of the iterates is established in \cite{Bian2020}. Many numerical algorithms based on the smoothing methods for solving the nonsmooth optimization problem have been studied extensively \cite{Chen2018A,Liu2016A,Xu2015S}.
More recently, on the basis of the regularity properties of the Moreau envelope of a proper convex lower semicontinuous function $f$, a class of APG algorithms \cite{Adly2020,Attouch2020C,Attouch2020F} for solving smooth convex optimization problems are extended to nonsmooth convex optimization problems. However, the resulting algorithms involve the proximal operator of $f$. We remark that the proximal operator of most objective functions doesn't have the closed-form solution, which isn't needed in the proposed algorithm of this paper.

For the numerical implementation of the algorithms, it is important to study the stability with respect to the computational errors or perturbations of the numerical algorithm. Being based on the well-posed dynamic systems with a small perturbation term, there are many references on the stability properties of the accelerated forward-backward algorithms with perturbation obtained by the implicit/explicit finite difference, such as \cite{Adly2020,Attouch-MP,Attouch-Peypouquet,Aujol,Villa}. It is worth emphasizing that the convergence results of these inexact algorithms are parallel to those of the algorithms in the unperturbed case under some condition on the perturbations and errors.
%On the other hand, it holds that the Moreau envelope $f_{\lambda}$ of a proper convex lower semicontinuous function $f$ is convex continuously differentiable, and the infimal value and the optimal solution set are preserved by taking the Moreau envelope. In view of these facts, the researchers replace $f$ by $f_{\lambda}$ in a class of inertial proximal algorithm.

Note that the study on solving (\ref{ob}) remains some serious difficulties and challenges that we now briefly sketch two points. First,
the global convergence rates of the APG method for (\ref{ob}) with a general (possible nonsmooth) convex function $c$ has not been extensively considered before. Though a good convergence rate on the objective function values is estimated in \cite{Nesterov2005}, any accumulation point of the algorithm in it is an $\epsilon$-solution, where $\epsilon$ is an appropriately selected fixed
positive parameter in the algorithm.
Second, as far as we know, the sequential convergence of APG methods for (\ref{ob}) in this case have not be proved before.
In this paper, we specialize the study of problem (\ref{ob}) in the case that $c$ is a nonsmooth convex function, and aim to
extend the results in \cite{Attouch-Peypouquet} to a more general composite convex minimization problem modeled by (\ref{ob}). Our main contributions are to
propose an APG algorithm for solving (\ref{ob}), which not only owns a fast global convergence rate on the objective values, but also possesses the sequence
convergence.
As a first attempt to design, we aim to introduce the efficient smoothing techniques into the APG method to overcome the nonsmoothness of $c$ in (\ref{ob}).
The challenge of improving the convergence rate of proposed algorithm is the updating method for the smoothing parameter.
After adjusting numerous times and learning the techniques on the convergence analysis in \cite{Attouch-MP,Attouch-Peypouquet}, we give an updating scheme of the smoothing
parameter, which not only let the proposed algorithm own the
global convergence rate of $o(\ln^{\sigma}k/k)$ with any $\sigma\in({1}/{2},1]$ on the objective function values, but also have the global sequential convergence.
In particular, we consider the effect of the errors in the proposed algorithm, and give an sufficient condition on the errors to guarantee the established convergence results of the algorithm without perturbation still hold.
We hope, this work will give some insight on improving the corresponding algorithms that need the Lipschitz continuous gradient as a basic assumption to
more general problems in applications.

The rest of this paper is organized as follows. Section \ref{section2} gives the definition of smoothing function and some necessary preliminary results, and lists several examples of modeling with \eqref{ob} in practical applications. In Section \ref{section3}, we design an accelerated algorithm, named smoothing accelerated proximal gradient (SAPG) algorithm, for solving problem \eqref{ob} and derive the main convergence results of it. We also discuss the stability of SPAG algorithm with respect to errors on the calculation of the gradient in this section. Section \ref{section4} demonstrates the performance of the proposed algorithm by some numerical experiments.

%\textbf{Notations:}
\paragraph{Notations:}
Throughout the paper, $\mathbb{R}^n$ is a $n$ dimensional Euclidean space equipped with the scalar product $\langle\cdot , \cdot\rangle$ and the Euclidean norm $\| \cdot \|$. We define $\mathbb{N}:=\{0, 1, 2, \cdots\}$. $\mathbb{R}_{+}$ denotes the set of all nonnegative real numbers. The notations $\mathbb{R}_{++}$ and $\| \cdot \|_1$ denote the set of all positive real numbers and the $\ell_1$ norm, respectively. For a vector $x\in \mathbb{R}^n$ and a nonempty closed convex set $\mathcal{X}\subseteq\mathbb{R}^n$, the projection operator to set $\mathcal{X}$ at $x$ is defined by $P_\mathcal{X}(x):=\arg\min\{\|x-z\|: z\in\mathcal{X}\}$.
\section{Preliminary results and examples}\label{section2}
As discussed in Introduction, the main difficulty in solving (\ref{ob}) by the PG and APG method is the nonsmoothness of $c$.
When $c$ is nonsmooth or $\nabla c$
is not globally Lipschitz continuous, an direct idea is to use the smoothing method, which plays a central role in our analysis. In this paper,
we will propose an algorithm with the smoothing function defined in \cite{Bian-Chen-SINMU}, which approximates the nonsmooth convex function $c$ by a class of smooth convex functions.
\begin{definition}\label{defn1}\cite{Bian-Chen-SINMU}
	For convex function $c$ in (\ref{ob}), we call $\tilde{c}:\mathbb{R}^n\times\mathbb{R}_{+}\rightarrow\mathbb{R}$ a smoothing function of $c$, if
	$\tilde{c}$ satisfies the following conditions:
	\begin{itemize}
		\item [{\rm (i)}] for any fixed $\mu>0$, $\tilde{c}(\cdot,\mu)$ is continuously differentiable on $\mathbb{R}^n$;
		\item [{\rm (ii)}] $\lim_{z\to x,\mu\downarrow0}\tilde{c}(z,\mu)=c(x)$, $\forall x\in\mathbb{R}^n$;
		\item [{\rm (iii)}] (gradient consistence) $\{\lim_{z\rightarrow x,\mu\downarrow0}\nabla_{z}\tilde{c}(z,\mu)\}\subseteq\partial c(x)$, $\forall x\in\mathcal{X}$;
		\item [{\rm (iv)}] for any fixed $\mu>0$, $\tilde{c}(\cdot,\mu)$ is convex on $\mathcal{X}$;
		\item [{\rm (v)}] there exists a $\kappa>0$ such that
		\begin{equation}\label{eq-s1}
		|\tilde{c}({{x}},\mu_2)-\tilde{c}({{x}},\mu_1)|\leq\kappa|\mu_1-\mu_2|,\quad \forall {{x}}\in\mathcal{X},\,\mu_1,\mu_2\in\mathbb{R}_{++};
		\end{equation}
		\item [{\rm (vi)}] there exists an $L>0$ such that $\nabla_x\tilde{c}({\cdot},\mu)$ is Lipschitz continuous on $\mathcal{X}$ with factor $L\mu^{-1}$
		for any fixed $\mu\in\mathbb{R}_{++}$.
	\end{itemize}
\end{definition}

Combining properties (ii) and (v) in Definition \ref{defn1}, we have
\begin{equation}\label{eq-s2}
|\tilde{c}(x,\mu)-c(x)|\leq\kappa\mu,\quad \forall {{x}}\in\mathcal{X},\,\mu\in\mathbb{R}_{++}.
\end{equation}

The study of smooth
approximations for various specialized nonsmooth functions has a long history and rich theoretical results \cite{Chen-MP,F-P-Book,Nesterov2005,Rock-Wets,HUL93}.
Items (i)-(iii) are basic conditions in the definition of smoothing function \cite{Chen-MP}, which are necessary for the effectiveness of the smoothing methods
in solving the corresponding nonsmooth problems. Item (iv) states that the smoothing function $\tilde{c}(\cdot,\mu)$ maintains the
convexity of $c$ for any fixed $\mu\in\mathbb{R}_{++}$.
Item (v) and (vi) ensure the global Lipschitz continuity of $\tilde{c}(x,\cdot)$ on $\mathbb{R}_{++}$ for any fixed $x\in\mathbb{R}^n$,
and the global Lipschitz continuity of $\nabla_x\tilde{c}(\cdot,\mu)$ for any fixed $\mu\in\mathbb{R}_{++}$, respectively.
\begin{remark}
	In particular, we want to mention in advance that the values of $\kappa$ and $L$ in Definition \ref{defn1} are not needed in the following proposed
	algorithm, but only used in the convergence analysis.
\end{remark}

Now, let us look at some examples in applications modeled by (\ref{ob}), for which we can construct a smoothing function for the first function and
calculate the proximal operator for the second function in its objective function. We can also refer to \cite{Nesterov2005} for some more examples in applications.
\begin{example}
	To find a sparse solution in $\mathcal{X}$ satisfying $c(x)\approx0$, one often considers the following $\ell_1$ regularized sparse optimization model
	\[\min_{x\in\mathcal{X}}\quad c(x)+\lambda\|x\|_1,\]
	where $c$ is the loss function to characterize the data fitting and $\lambda>0$ is the penalty parameter.
	Notice that the proximal operator of $\ell_1$ function has a closed form expression \cite{Parikh}.
	A notable nonsmooth convex loss function in linear regression problem is the $\ell_1$ function, which takes the following form
	\begin{equation}\label{l1-l1}
	c(x):=\|Ax-b\|_1,
	\end{equation}
	with $A\in\mathbb{R}^{m\times n}$ and $b\in\mathbb{R}^m$.
	As pointed out in \cite{Fan2001}, $\ell_1$ loss function is nonsmooth, but more robust and has stronger capability of outlier-resistant
	than the least square loss function in the linear regression problems. Another important application is the loss function in censored
	regression problem, which is often in the form of
	\begin{equation}\label{max-l1}
	c(x):=\left\|\max\{Ax,0\}-b\right\|_q^q
	\end{equation}
	with $q\in[1,2]$. Function $c$ in \eqref{max-l1} is also a convex but nonsmooth function.
	Besides, both the check loss function in penalized quantile regression \cite{Fan2014,Koenker} and the negative log-quasi-likehood
	loss function \cite{Fan2001} are nonsmooth convex functions.
	One can consult \cite{Bian-Chen-SINMU} for the construction of smoothing functions satisfying Definition \ref{defn1}, including the
	smoothing functions of $c$ in (\ref{l1-l1}) and (\ref{max-l1}).
\end{example}
\begin{example}
	Since $g$ can be a possible nonsmooth extended valued function, we can let $g$ be the indicator function of a closed convex subset
	$\mathcal{Y}$ of $\mathbb{R}^n$, i.e.
	\[\delta_{\mathcal{Y}}(x)=\left\{\begin{aligned}
	&0&&\mbox{if $x\in\mathcal{Y}$}\\
	&\infty&&\mbox{if $x\not\in\mathcal{Y}$}.
	\end{aligned}\right.\]
	Here, $\delta_{\mathcal{Y}}$ is a proper, lower semicontinuous convex function, and ${\rm prox}_{\vartheta\delta_{\mathcal{Y}}}$ is the
	projection operator onto $\mathcal{Y}$.
	Then, problem (\ref{ob}) is reduced to the following constrained (maybe nonsmooth) convex minimization problem
	\[\min_{x\in\mathcal{X}\cap\mathcal{Y}}\quad c(x).\]
\end{example}
\begin{example} Consider the following constrained convex optimization problem
	\begin{equation}\label{ex3}
	\begin{aligned}
	&\min\quad &&g(x)\\
	&\mbox{\,\,s.t.}&&x\in\mathcal{X},\,h_i(x)\leq0,\,i=1,\ldots,r,
	\end{aligned}
	\end{equation}
	where $g$ and $h_i:\mathbb{R}^n\rightarrow\mathbb{R}$, $i=1,\ldots,r$, are convex functions. It is well known that most exact penalty functions are nonsmooth.
	Based on the exact penalty method, under some proper conditions \cite{Clarke,Nocedal,Bian-Xue-CS}, problem (\ref{ex3}) can be equivalent to
	\begin{equation}\label{ex3-2}
	\begin{aligned}
	&\min\quad &&\lambda\sum_{i=1}^r\max\{h_i(x),0\}+g(x)\\
	&\mbox{\,\,s.t.}&&x\in\mathcal{X},
	\end{aligned}
	\end{equation}
	where $\lambda$ is an exact penalty parameter. Problem (\ref{ex3-2}) is also a special case of (\ref{ob}) with
	\[c(x):=\lambda\sum_{i=1}^r\max\{h_i(x),0\},\]
	which is a nonsmooth convex function and can have a smoothing function satisfying Definition \ref{defn1}.
\end{example}

\section{SAPG algorithm and convergence analysis}\label{section3}
As inspired by the success in \cite{Attouch-Peypouquet}, we will propose an accelerated algorithm for solving (\ref{ob})
based on the scheme of (\ref{Attouch}) with $\beta_k=\frac{k-1}{k+\alpha-1}$ and $\alpha>3$. Equipped with the smoothing
function of $c$ defined in Definition \ref{defn1}, we can now develop an accelerated proximal gradient algorithm for (\ref{ob})
and build the global convergence analysis including the fast convergence rate on the objective function values and the sequential convergence of iterates.
In particular, it owns the global convergence rate of $o(\ln^{\sigma}k/k)$ with any $\sigma\in(\frac{1}{2},1]$.

For easy of reference and correspond to its structure, we call the proposed algorithm smoothing accelerated proximal gradient (SAPG) algorithm in this paper. In this section, we always let $\tilde{c}$ a smoothing function of $c$ defined in Definition \ref{defn1} with positive parameters $\kappa$ and $L$ in item (v) and (vi), respectively.
In what follows,
$\nabla \tilde{c}(x,\mu)$ means the gradient of $\tilde{c}$ with respect to $x$ for simplicity.

\subsection{SAPG algorithm}
Set
\begin{equation}\label{sob}
\tilde{f}(x,\mu):=\tilde{c}(x,\mu)+g(x).
\end{equation}
Notice that $\tilde{f}$ can be a nonsmooth function, since we do not assume the smoothness of $g$.
However, for any fixed $\mu>0$, $\tilde{f}(\cdot,\mu)$ is with the composite structure, whose first term is a Lipschitz continuously
differentiable convex function and the second term is a proper lower semicontinuous convex function with computable proximal operator.
This provides a possibility
of adopting the PG and APG method on it.
Different from the work in \cite{Nesterov2005}, which is for problem $\min_{x\in\mathcal{X}}\tilde{f}(x,\mu)$
with a fixed and appropriate selected value of $\mu$, we will let the parameter $\mu$ tend to $0$ asymptotically and give a global convergence rate on the objective function
values. The updating method
of $\mu$ plays a key role in the global convergence analysis and affects the convergence rate directly. We are now ready to present the proposed SAPG
algorithm for solving (\ref{ob}). See Algorithm \ref{algorithm1}.

\begin{algorithm}[!h]
	\caption{Smoothing Accelerated Proximal Gradient (SAPG) Algorithm }\label{algorithm1}
	\begin{itemize}
		\item[\textbf{Input:}]
		Take initial point $x^{-1}=x^0\in\mathcal{X}$, $\mu_0\in\mathbb{R}_{++}$ and $\gamma_0>0$. Choose parameters $\eta\in(0,1)$, $\alpha>3$ and $\sigma\in\left(\frac{1}{2}, 1\right]$.
		Set $k=0$.
		\item[\textbf{Step 1:}] Set $\hat{\gamma}_{k+1}=\gamma_k$ and compute
		\begin{equation}\label{algo1}
		y^k=x^k+\frac{k-1}{k+\alpha-1}(x^k-x^{k-1}) ,
		\end{equation}
		\begin{equation}\label{algo2}
		\mu_{k+1}=\frac{\mu_0}{(k+\alpha-1)\ln^{\sigma}(k+\alpha-1)}.
		\end{equation}
		\item[\textbf{Step 2:}] Compute
		\begin{equation}\label{algo3}
		\hat{x}^{k+1}={\rm prox}_{\hat{\gamma}_{k+1}{\mu_{k+1}}g}(y^k-\hat{\gamma}_{k+1}\mu_{k+1}\nabla\tilde{c}(y^k,\mu_{k+1})). \end{equation}
		\item[\textbf{Step 3:}] If $\hat{x}^{k+1}$ satisfies
		\begin{equation}\label{algo4}
		\begin{aligned}
		\tilde{c}(\hat{x}^{k+1},\mu_{k+1})
		\leq\tilde{c}({y}^{k},\mu_{k+1})+&\langle\nabla\tilde{c}({y}^{k},\mu_{k+1}),
		\hat{x}^{k+1}-y^k\rangle\\
		&+\frac{1}{2}(\hat{\gamma}_{k+1}\mu_{k+1})^{-1}\|\hat{x}^{k+1}-y^{k}\|^2,
		\end{aligned}\end{equation}
		let $\gamma_{k+1}=\hat{\gamma}_{k+1}$, $x^{k+1}=\hat{x}^{k+1}$, increment $k$ by one and return to \textbf{Step 1}.\\
		Otherwise, let $\hat{\gamma}_{k+1}=\eta\hat{\gamma}_{k+1}$ and return to \textbf{Step 2}.\\
	\end{itemize}
	
\end{algorithm}
It is interesting to see that (\ref{algo1}) and (\ref{algo3}) are the iterations of general APG method for
$\tilde{f}(x,\mu_{k+1})$, which indicates that the SAPG algorithm shares the same structural decomposition principle
and extrapolation as the usual APG algorithm in \cite{Attouch-Peypouquet}.
Similar as the work in \cite{Bian-Chen-SINMU}, (\ref{algo4}) is a simple line search for verifying the adaptability
of $(\hat{\gamma}_{k+1}\mu_{k+1})^{-1}$ on the Lipschitz constant of $\nabla \tilde{c}(\cdot,\mu_{k+1})$ between $y^k$
and $\hat{x}^{k+1}$. It should be carefully noted that smoothing parameter $\{\mu_k\}$ is strictly monotone decreasing and tends to $0$ as $k\to\infty$ if the SAPG algorithm is well-defined for all $k\in\mathbb{N}$.

\subsection{Some basic estimations}
We start with some basic preliminary estimations for the SAPG algorithm.\
Let $\{x^k\}$, $\{y^k\}$, $\{\gamma_k\}$ and $\{\mu_k\}$ be the sequences generated by the SAPG algorithm.
Our convergence analysis follows some basic ideas in \cite{Attouch-Peypouquet} and extends it to a more general nonsmooth case.

Set
$$Q(x,y,\mu,\gamma):=\tilde{c}(y,\mu)+\langle\nabla\tilde{c}(y,\mu),x-y\rangle+\frac{1}{2}(\gamma\mu)^{-1}\|x-y\|^2+g(x).$$
For fixed $y$, $\mu$ and $\gamma$, $Q(\cdot,y,\mu,\gamma)$ is strongly convex with modulus $(\gamma\mu)^{-1}$, and admits a
unique global minimizer on $\mathcal{X}$, which is denoted by $p(y,\mu,\gamma)$, i.e.
$$p(y,\mu,\gamma):=\arg\min_{x\in\mathcal{X}}Q(x,y,\mu,\gamma).$$
Then,
it ensures that
\begin{equation}\label{eq2-1}
Q(x,y,\mu,\gamma)\geq Q(p(y,\mu,\gamma),y,\mu,\gamma)+\frac{1}{2}(\gamma\mu)^{-1}\|x-p(y,\mu,\gamma)\|^2,\quad \forall x\in\mathcal{X}.
\end{equation}

Invoking the formulation of proximal operator in (\ref{prox}), (\ref{algo3}) can be expressed as
\begin{equation}\label{eq-p}
\hat{x}^{k+1}=p(y^k,\mu_{k+1},\gamma_{k+1}).
\end{equation}
\begin{lemma}\label{lemma1}
	The SAPG algorithm is well-defined. Moreover, the sequences $\{x^k\}$, $\{\gamma_k\}$ and $\{\mu_k\}$ satisfy
	\begin{itemize}
		\item[{\rm(i)}] $\{\gamma_k\}$ is non-increasing and lower bounded by $\underline{\gamma}:=\min\{{\gamma}_0,\eta L^{-1}\}$;
		\item[{\rm(ii)}] $\{\mu_k\}$ is monotone decreasing and $\lim_{k\rightarrow\infty}\mu_k=0$;
		\item[{\rm(iii)}] for any $k\geq0$, $x^k\in\mathcal{X}$.
	\end{itemize}
\end{lemma}
\begin{proof}
	Invoking Definition \ref{defn1}-(vi), (\ref{algo4}) holds when $\gamma_{k+1}^{-1}\geq L$, which implies that the
	total updating iterations for returning to step 2 is at most $1+[-\log_{\eta}(L{\gamma}_0)]$ times, where $[a]$ means the largest
	positive integer such that $[a]\leq a$. Combining this with (\ref{eq-p}), the SAPG algorithm is well-defined
	and that $\gamma_k\geq\underline{\gamma}:=\min\{{\gamma}_0,\eta L^{-1}\}$.
	
	From (\ref{algo2}) and (\ref{algo3}), we can directly verify the results in items (ii) and (iii).
\end{proof}

\begin{lemma}\label{lemma2}
	For any $x\in\mathcal{X}$ and $k\in\mathbb{N}$, it holds that
	\begin{equation}\label{lem2-3}
	\begin{aligned}
	\tilde{f}(x^{k+1},\mu_{k+1})\leq&\tilde{f}(x,\mu_{k+1})
	+(\gamma_{k+1}\mu_{k+1})^{-1}\langle y^k-x^{k+1},y^k-x\rangle\\
	&-\frac{1}{2}(\gamma_{k+1}\mu_{k+1})^{-1}\|x^{k+1}-y^{k}\|^2.
	\end{aligned}\end{equation}
\end{lemma}
\begin{proof}
	Letting $y=y^k$, $\mu=\mu_{k+1}$ and $\gamma=\gamma_{k+1}$ in (\ref{eq2-1}), we have
	\[\begin{aligned}
	&Q(x,y^k,\mu_{k+1},\gamma_{k+1})\\
	\geq&
	Q(p(y^k,\mu_{k+1},\gamma_{k+1}),y^k,\mu_{k+1},\gamma_{k+1})
	+\frac{1}{2}(\gamma_{k+1}\mu_{k+1})^{-1}\|x-p(y^k,\mu_{k+1},\gamma_{k+1})\|^2\\
	=&Q(x^{k+1},y^k,\mu_{k+1},\gamma_{k+1})
	+\frac{1}{2}(\gamma_{k+1}\mu_{k+1})^{-1}\|x-x^{k+1}\|^2,\quad\forall x\in\mathcal{X}.
	\end{aligned}\]
	Upon rearranging the terms, we deduce that, for any $x\in\mathcal{X}$,
	\begin{equation}\label{lem2-1}
	\begin{aligned}
	g(x^{k+1})\leq& g(x)+\langle\nabla\tilde{c}(y^k,\mu_{k+1}),x-x^{k+1}\rangle+
	\frac{1}{2}(\gamma_{k+1}\mu_{k+1})^{-1}\|x-y^{k}\|^2\\
	&-\frac{1}{2}(\gamma_{k+1}\mu_{k+1})^{-1}\|x-x^{k+1}\|^2
	-\frac{1}{2}(\gamma_{k+1}\mu_{k+1})^{-1}\|x^{k+1}-y^{k}\|^2.
	\end{aligned}\end{equation}
	Recalling (\ref{algo4}) with $\gamma_{k+1}=\hat{\gamma}_{k+1}$ and $x^{k+1}=\hat{x}^{k+1}$, we find
	\begin{equation}\label{lem2-2}
	\begin{aligned}
		&\tilde{c}(x^{k+1},\mu_{k+1}) \\
		\leq &\tilde{c}(y^{k},\mu_{k+1})+
	\langle\nabla\tilde{c}(y^{k},\mu_{k+1}),x^{k+1}-y^k\rangle+\frac{1}{2}(\gamma_{k+1}\mu_{k+1})^{-1}\|x^{k+1}-y^k\|^2.
	\end{aligned}
	\end{equation}
	Adding (\ref{lem2-1}) and (\ref{lem2-2}) together, we deduce that, for any $x\in\mathcal{X}$,
	\begin{equation}\label{lem2-4}
	\begin{aligned}
	&\tilde{f}(x^{k+1},\mu_{k+1})=\tilde{c}(x^{k+1},\mu_{k+1})+g(x^{k+1})\\
	\leq&\tilde{c}(y^{k},\mu_{k+1})+
	\langle\nabla\tilde{c}(y^{k},\mu_{k+1}),x-y^k\rangle
	+\frac{1}{2}(\gamma_{k+1}\mu_{k+1})^{-1}\|x-y^{k}\|^2\\
	&-\frac{1}{2}(\gamma_{k+1}\mu_{k+1})^{-1}\|x-x^{k+1}\|^2+g(x)\\
	\leq&\tilde{c}(x,\mu_{k+1})+g(x)+\frac{1}{2}(\gamma_{k+1}\mu_{k+1})^{-1}\|x-y^{k}\|^2
	-\frac{1}{2}(\gamma_{k+1}\mu_{k+1})^{-1}\|x-x^{k+1}\|^2 \\
	=&\tilde{f}(x,\mu_{k+1})+\frac{1}{2}(\gamma_{k+1}\mu_{k+1})^{-1}\|x-y^{k}\|^2
	-\frac{1}{2}(\gamma_{k+1}\mu_{k+1})^{-1}\|x-x^{k+1}\|^2,
	\end{aligned}
	\end{equation}
	where the second inequality follows from the convexity of $\tilde{c}(\cdot,\mu_{k+1})$ assumed in Definition \ref{defn1}-(iv).
	Thus, the rearranging terms of (\ref{lem2-4}) gives (\ref{lem2-3}).
\end{proof}

Fix $x^*\in\arg\min_{x\in\mathcal{X}}f$. Let us bring forward the global energy function that will serve for the Lyapunov analysis:
\begin{equation}\label{prop1-10}
\begin{aligned}
\mathcal{E}_k:=&\frac{2\gamma_{k}\mu_{k}}{\alpha-1}(k+\alpha-2)^2W_k
+(\alpha-1)\|u^{k}-x^*\|^2\\
&+\left(\frac{4\kappa{\gamma}_0\mu_0}{2\sigma-1}\right)\mu_k(k+\alpha-2)\ln^{1-\sigma}(k+\alpha-2),
\end{aligned}\end{equation}
where
\begin{equation}\label{prop1-11}
W_k:=\tilde{f}(x^{k},\mu_{k})+\kappa\mu_{k}-f(x^*)\quad\mbox{and}\quad u^k:=\left(\frac{k+\alpha-2}{\alpha-1}\right)x^{k}-\left(\frac{k-1}{\alpha-1}\right)x^{k-1}.
\end{equation}
The following proposition gives a key estimation for the forthcoming analysis. It provides the most important thing that
$\{\mathcal{E}_k\}$ is non-increasing for all $k$.
\begin{proposition}\label{prop1}
	Let $\mathcal{E}_k$ be the sequence defined in (\ref{prop1-10}). Then, for any $k\geq1$, we have
	\begin{equation}\label{prop1-9}
	\mathcal{E}_{k+1}+\frac{2(\alpha-3)\gamma_{k+1}\mu_{k+1}}{\alpha-1}(k+\alpha-1)W_k\leq\mathcal{E}_{k}.
	\end{equation}
	Moreover,
	\begin{itemize}
		\item[{\rm(i)}]
		the sequence $\{\mathcal{E}_k\}$ is non-increasing for all $k\geq1$, and $\lim_{k\rightarrow\infty}\mathcal{E}_k$ exists;
		\item[{\rm(ii)}] for every $k\geq1$, $$\mathcal{E}_k\leq(\alpha-1)\|x^*-x^0\|^2+4(\alpha-1)\kappa\gamma_{0}\mu_{0}^2
		+\frac{4\kappa{\gamma}_0\mu_0^2}{2\sigma-1}(\alpha-1)\ln^{1-\sigma}(\alpha-1).$$
	\end{itemize}
\end{proposition}
\begin{proof}
	(i). Let us write inequality (\ref{lem2-3}) at $x=x^k$ and $x=x^*$, respectively, i.e.
	\begin{equation}\label{prop2-7}
	\begin{aligned}
	\tilde{f}(x^{k+1},\mu_{k+1})\leq&\tilde{f}(x^k,\mu_{k+1})
	+(\gamma_{k+1}\mu_{k+1})^{-1}\langle y^k-x^{k+1},y^k-x^k\rangle\\
	&-\frac{1}{2}(\gamma_{k+1}\mu_{k+1})^{-1}\|x^{k+1}-y^{k}\|^2
	\end{aligned}
	\end{equation}
	and
	\begin{equation}\label{prop2-8}
	\begin{aligned}
	\tilde{f}(x^{k+1},\mu_{k+1})\leq&\tilde{f}(x^*,\mu_{k+1})
	+(\gamma_{k+1}\mu_{k+1})^{-1}\langle y^k-x^{k+1},y^k-x^*\rangle\\
	&-\frac{1}{2}(\gamma_{k+1}\mu_{k+1})^{-1}\|x^{k+1}-y^{k}\|^2.
	\end{aligned}
	\end{equation}
	Substituting the algebraic inequalities (\ref{eq-s1}) and (\ref{eq-s2}) into (\ref{prop2-7}) and  (\ref{prop2-8}), respectively,
	we have the following two inequalities
	\begin{equation}\label{prop1-1}
	\begin{aligned}
	W_{k+1}\leq &W_{k}
	+(\gamma_{k+1}\mu_{k+1})^{-1}\langle y^k-x^{k+1},y^k-x^k\rangle\\
	&-\frac{1}{2}(\gamma_{k+1}\mu_{k+1})^{-1}\|x^{k+1}-y^{k}\|^2
	\end{aligned}\end{equation}
	and
	\begin{equation}\label{prop1-2}
	\begin{aligned}
	W_{k+1}\leq&(\gamma_{k+1}\mu_{k+1})^{-1}\langle y^k-x^{k+1},y^k-x^*\rangle\\
	&-\frac{1}{2}(\gamma_{k+1}\mu_{k+1})^{-1}\|x^{k+1}-y^{k}\|^2+2\kappa\mu_{k+1}.
	\end{aligned}\end{equation}
	
	Multiplying (\ref{prop1-1}) by $\frac{k}{k+\alpha-1}$, and (\ref{prop1-2}) by $\frac{\alpha-1}{k+\alpha-1}$, then adding them together, we obtain that
	\begin{equation}\label{prop1-3}
	\begin{aligned}
	W_{k+1}
	\leq & \frac{k}{k+\alpha-1}W_k
	+\frac{2\kappa(\alpha-1)}{k+\alpha-1}\mu_{k+1}
	-\frac{1}{2}(\gamma_{k+1}\mu_{k+1})^{-1}\|x^{k+1}-y^{k}\|^2\\
	&+
	(\gamma_{k+1}\mu_{k+1})^{-1}\left\langle y^k-x^{k+1},\frac{k}{k+\alpha-1}(y^k-x^k)+
	\frac{\alpha-1}{k+\alpha-1}(y^k-x^*)
	\right\rangle.
	\end{aligned}\end{equation}
	By using the algebraic inequality
	\begin{equation}\label{prop1-14}
	-\|a-b\|^2+2\langle b-a,b-c\rangle=-\|a-c\|^2+\|b-c\|^2
	\end{equation}
	with $a=x^{k+1}$, $b=y^k$ and $c=\frac{k}{k+\alpha-1}x^k+\frac{\alpha-1}{k+\alpha-1}x^*$, we observe that
	\begin{equation}\label{prop1-4}
	\begin{aligned}
	&-\|x^{k+1}-y^k\|^2+2\left\langle y^k-x^{k+1},y^k-\frac{k}{k+\alpha-1}x^k-\frac{\alpha-1}{k+\alpha-1}x^*\right\rangle\\
	=&-\left\|x^{k+1}-\frac{k}{k+\alpha-1}x^k-\frac{\alpha-1}{k+\alpha-1}x^*\right\|^2+
	\left\|y^k-\frac{k}{k+\alpha-1}x^k-\frac{\alpha-1}{k+\alpha-1}x^*\right\|^2\\
	%=&-\|x^{k+1}-\frac{k}{k+\alpha-1}x^k+\frac{\alpha-1}{k+\alpha-1}x^*\|^2+
	%\|\frac{k+\alpha-2}{k+\alpha-1}x^k-\frac{k-1}{k+\alpha-1}x^{k-1}-\frac{\alpha-1}{k+\alpha-1}x^*\|^2\\
	%=&-\left(\frac{\alpha-1}{k+\alpha-1}\right)^2\left(\left\|\frac{k+\alpha-1}{\alpha-1}x^{k+1}-\frac{k}{\alpha-1}x^k-x^*\right\|^2
	%-\left\|\frac{k+\alpha-2}{\alpha-1}x^{k}-\frac{k-1}{\alpha-1}x^{k-1}-x^*\right\|^2
	%\right)\\
	=&-\left(\frac{\alpha-1}{k+\alpha-1}\right)^2\left(\left\|u^{k+1}-x^*\right\|^2
	-\left\|u^k-x^*\right\|^2\right),
	\end{aligned}
	\end{equation}
	where the last equality uses the expression of $y^k$ in (\ref{algo1}).
	Substituting (\ref{prop1-4}) into (\ref{prop1-3}) and by simple algebraic manipulations, we have
	\[\begin{aligned}
	W_{k+1}\leq&\frac{k}{k+\alpha-1}W_k+\frac{2\kappa(\alpha-1)}{k+\alpha-1}\mu_{k+1}\\
	&-\frac{1}{2}(\gamma_{k+1}\mu_{k+1})^{-1}\left(\frac{\alpha-1}{k+\alpha-1}\right)^2\left(\|u^{k+1}-x^*\|^2
	-\|u^k-x^*\|^2\right).
	\end{aligned}
	\]
	Multiplying the above inequality by $\frac{2\gamma_{k+1}\mu_{k+1}}{\alpha-1}(k+\alpha-1)^2$, we obtain
	\begin{equation}\label{prop1-5}
	\begin{aligned}
	\frac{2\gamma_{k+1}\mu_{k+1}}{\alpha-1}(k+\alpha-1)^2W_{k+1}
	\leq&\frac{2\gamma_{k+1}\mu_{k+1}}{\alpha-1}k(k+\alpha-1)W_k
	+4\kappa({k+\alpha-1})\gamma_{k+1}\mu_{k+1}^2\\
	&-(\alpha-1)\left(\|u^{k+1}-x^*\|^2
	-\|u^k-x^*\|^2\right).
	\end{aligned}\end{equation}
	Since $\alpha>3$, we infer that
	\begin{equation}\label{prop1-6}
	\begin{aligned}
	k(k+\alpha-1)=&(k+\alpha-2)^2-(\alpha-3)(k+\alpha-1)-1 \\
	<&(k+\alpha-2)^2-(\alpha-3)(k+\alpha-1).
	\end{aligned}
	\end{equation}
	Then, reformulating (\ref{prop1-5}) by (\ref{prop1-6}) yields
	\begin{equation}\label{prop1-7}
	\begin{aligned}
	&\frac{2\gamma_{k+1}\mu_{k+1}}{\alpha-1}(k+\alpha-1)^2W_{k+1}+\frac{2(\alpha-3)\gamma_{k+1}\mu_{k+1}}{\alpha-1}(k+\alpha-1)W_{k}\\
	&\quad 	+(\alpha-1)\|u^{k+1}-x^*\|^2 \\
	\leq&\frac{2\gamma_{k+1}\mu_{k+1}}{\alpha-1}(k+\alpha-2)^2W_k
	+(\alpha-1)\|u^{k}-x^*\|^2
	+4\kappa({k+\alpha-1})\gamma_{k+1}\mu_{k+1}^2\\
	\leq&\frac{2\gamma_{k}\mu_{k}}{\alpha-1}(k+\alpha-2)^2W_k
	+(\alpha-1)\|u^{k}-x^*\|^2
	+4\kappa{\gamma}_0({k+\alpha-1})\mu_{k+1}^2,
	\end{aligned}
	\end{equation}
	where the last inequality uses the non-increasing property of $\{\gamma_k\}$ and $\{\mu_k\}$.
	To write (\ref{prop1-7}) in a recursive form, by the definition of $\mu_{k}$ given in (\ref{algo2}), we observe that
	\begin{equation}\label{prop1-8}
	\begin{aligned}
	&\mu_0\mu_k(k+\alpha-2)\ln^{1-\sigma}(k+\alpha-2)-\mu_0\mu_{k+1}(k+\alpha-1)\ln^{1-\sigma}(k+\alpha-1)\\
	%&\mu_0\frac{(k+\alpha-2)\mu_k}{\ln^{\sigma-1}(k+\alpha-2)}-\mu_0\frac{(k+\alpha-1)\mu_{k+1}}{\ln^{\sigma-1}(k+\alpha-1)}\\
	=&\mu_0^2
	\left(\ln^{1-2\sigma}(k+\alpha-2)-\ln^{1-2\sigma}(k+\alpha-1)\right)\\
	%=&\frac{\mu_0^2(k+\alpha-1)\ln(k+\alpha-1)}{(k+\alpha-1)\ln^{2\sigma}(k+\alpha-1)}
	%\left(\frac{\ln^{2\sigma-1}(k+\alpha-1)}{\ln^{2\sigma-1}(k+\alpha-2)}-1\right)\\
	=&\frac{\mu_0^2}{\ln^{2\sigma-1}(k+\alpha-1)}
	\left(\frac{\ln^{2\sigma-1}(k+\alpha-1)}{\ln^{2\sigma-1}(k+\alpha-2)}-1\right)
	\\
	\geq&(2\sigma-1)(k+\alpha-1)\mu_{k+1}^2\frac{\ln^{2\sigma-1}(k+\alpha-1)}{\ln^{2\sigma-1}(k+\alpha-2)}\\
	\geq&(2\sigma-1)(k+\alpha-1)\mu_{k+1}^2,
	\end{aligned}
	\end{equation}
	where the third inequality follows from
	$$\ln^{2\sigma-1}(k+\alpha-1)\geq\ln^{2\sigma-1}(k+\alpha-2)
	+(2\sigma-1)\frac{\ln^{2\sigma-2}(k+\alpha-1)}{(k+\alpha-1)},\quad \forall\sigma\in\left(\frac{1}{2}, 1\right] .$$
	
	Combining (\ref{prop1-7}) with (\ref{prop1-8}), we finish the proof for the estimation in (\ref{prop1-9}).
	Thus, $\{\mathcal{E}_k\}$ is non-increasing for $k\geq1$, i.e. $\mathcal{E}_k\leq\mathcal{E}_1$.
	
	By (\ref{eq-s2}) and $x^k\in\mathcal{X}$, we get
	\begin{equation}\label{prop1-15}
	W_k=\tilde{f}(x^{k},\mu_{k})+\kappa\mu_{k}-f(x^*)\geq f(x^k)-f(x^*)\geq0,
	\end{equation}
	then $\mathcal{E}_k$ is lower bounded by $0$, which implies the existence of
	$\lim_{k\rightarrow\infty}\mathcal{E}_k$.
	
	(ii). According to the non-increasing of sequence $\{\mathcal{E}_k\}$, then, all that remains is to estimate the upper bound of
	\begin{equation}\label{prop1-13}
	\mathcal{E}_1=2\gamma_{1}\mu_{1}(\alpha-1)W_1
	+(\alpha-1)\|u^{1}-x^*\|^2
	+\frac{4\kappa{\gamma}_0\mu_0}{2\sigma-1}\mu_1(\alpha-1)\ln^{1-\sigma}(\alpha-1).
	\end{equation}
	Returning to (\ref{prop1-2}) with $k=0$, and using \eqref{prop1-14},
	%following the similar deduction of (\ref{prop1-4}),
	we deduce that
	\[\begin{aligned}
	W_1\leq &(\gamma_{1}\mu_{1})^{-1}\langle y^0-x^{1},y^0-x^*\rangle
	-\frac{1}{2}(\gamma_{1}\mu_{1})^{-1}\|x^{1}-y^{0}\|^2+2\kappa\mu_{1}\\
	= & \frac{1}{2}(\gamma_{1}\mu_{1})^{-1}\|x^*-y^{0}\|^2
	-\frac{1}{2}(\gamma_{1}\mu_{1})^{-1}\|x^*-x^{1}\|^2+2\kappa\mu_{1},
	\end{aligned}
	\]
	which implies
	\begin{equation}\label{prop1-12}
	2\gamma_{1}\mu_{1}W_1
	+\|x^*-x^{1}\|^2
	\leq\|x^*-y^{0}\|^2+4\kappa\gamma_{1}\mu_{1}^2.
	\end{equation}
	Recalling the definition of $u^1$ in (\ref{prop1-11}), we see that
	$u^1=x^{1}$, combining which with (\ref{prop1-13}), (\ref{prop1-12}), $y^0=x^0$, $\mu_1<\mu_0$ and $\gamma_1\leq\gamma_0$, we find that
	\begin{equation}\label{prop1-16}
	\mathcal{E}_1\leq(\alpha-1)\|x^*-x^0\|^2+4(\alpha-1)\kappa\gamma_{0}\mu_{0}^2+\frac{4\kappa{\gamma}_0\mu_0^2}{2\sigma-1}(\alpha-1)\ln^{1-\sigma}(\alpha-1).
	\end{equation}
	Hence, we establish the evaluation in item (ii).
\end{proof}

As a result of Proposition \ref{prop1}, we obtain some important properties of $W_k$ as shown below, where we need introduce an important lemma on sequence convergence.
\begin{lemma}\cite{Attouch-MP}\label{add2}
	Let $\{a_k\}$ be a sequence of nonnegative numbers, and satisfy
	\begin{equation*}
		\sum_{k=1}^{\infty} (a_{k+1}-a_k)_{+}<\infty.
	\end{equation*}
	Then, $\lim_{k\to\infty}a_k$ exists.
\end{lemma}

\begin{proposition}\label{prop2}
	Let $\{W_k\}$ be the sequence defined in \eqref{prop1-11}, then
	\begin{itemize}
		\item[{\rm(i)}] $\sum_{k=1}^{\infty}\gamma_{k}\mu_{k}(k+\alpha-2)W_k<\infty$;
		\item[{\rm(ii)}] $\lim_{k\rightarrow\infty}\left[(k-1)^2\|x^{k}-x^{k-1}\|^2+2\gamma_{k}\mu_{k}(k+\alpha-2)^2W_{k}\right]$ exists;
		\item[{\rm(iii)}] $\sum_{k=1}^{\infty}(k-1)\|x^{k}-x^{k-1}\|^2<\infty$.
	\end{itemize}
\end{proposition}
\begin{proof}
	(i). By summing up inequality (\ref{prop1-9}) from $k=1$ to $K$, we see that
	\[
	\mathcal{E}_{K+1}+\frac{2(\alpha-3)}{\alpha-1}\sum_{k=1}^K\gamma_{k+1}\mu_{k+1}(k+\alpha-1)
	W_k\leq\mathcal{E}_{1}.
	\]
	After letting $K$ tend to infinity in the above inequality and using (\ref{prop1-16}), since $\alpha>3$ and $\mathcal{E}_k\geq0$ for all $k\geq0$, we infer that
	\begin{equation}\label{prop2-3}
	\begin{aligned}
	\sum_{k=1}^{\infty}\gamma_{k+1}\mu_{k+1}(k+\alpha-1)W_k
	\leq\frac{(\alpha-1)\mathcal{E}_1}{2(\alpha-3)}<\infty.
	\end{aligned}
	\end{equation}
	Since, for all $k\geq 1$, it holds that
	\begin{equation}\label{pro_add}
	\begin{aligned}
	\gamma_{k+1}\mu_{k+1}(k+\alpha-1)=\gamma_{k+1}\frac{\mu_0}{\ln^{\sigma}(k+\alpha-1)}
	\geq&\underline{\gamma}\mu_k\frac{(k+\alpha-2)\ln^{\sigma}(k+\alpha-2)}{\ln^{\sigma}(k+\alpha-1)}\\
	\geq&\frac{\underline{\gamma}\ln^{\sigma}(\alpha-1)}{\gamma_0\ln^{\sigma}\alpha}\gamma_k\mu_k(k+\alpha-2)
	,%\,\,\forall k\geq1,
	\end{aligned}
	\end{equation}
	which uses $\underline{\gamma}\leq\gamma_k\leq\gamma_0$ proved in Lemma \ref{lemma1} and the increasing of $\frac{\ln^{\sigma}(k+\alpha-2)}{\ln^{\sigma}(k+\alpha-1)}$ for all $k\geq1$, then (\ref{prop2-3}) implies
	\[\sum_{k=1}^{\infty}\gamma_{k}\mu_{k}(k+\alpha-2)W_k<\infty.
	\]
	
	(ii). Returning to (\ref{prop1-1}) and using (\ref{prop1-14}) with $a=x^{k+1}$, $b=y^k$ and $c=x^k$, we deduce that
	\[
	W_{k+1}\leq W_{k}
	-\frac{1}{2}(\gamma_{k+1}\mu_{k+1})^{-1}\|x^{k+1}-x^{k}\|^2+
	\frac{1}{2}(\gamma_{k+1}\mu_{k+1})^{-1}\|y^{k}-x^{k}\|^2.
	\]
	In view of the definition of $y^k$ in (\ref{algo1}), we have from multiplying $2\gamma_{k+1}\mu_{k+1}(k+\alpha-1)^2$ on  the above inequality that
	\begin{equation}\label{prop2-4}
	\begin{aligned}
	&2\gamma_{k+1}\mu_{k+1}(k+\alpha-1)^2W_{k+1}+(k+\alpha-1)^2\|x^{k+1}-x^{k}\|^2\\
	\leq&2\gamma_{k+1}\mu_{k+1}(k+\alpha-1)^2W_{k}+(k-1)^2\|x^k-x^{k-1}\|^2.
	\end{aligned}
	\end{equation}
	By $k+\alpha-1\geq k$ and upon rearranging terms of the above inequality, it suffices to observe that
	\begin{equation}\label{prop2-1}
	\begin{aligned}
0\geq &k^2\|x^{k+1}-x^{k}\|^2 -(k-1)^2\|x^{k}-x^{k-1}\|^2\\
&+2\gamma_{k+1}\mu_{k+1}(k+\alpha-1)^2(W_{k+1}-W_{k}).
\end{aligned}
	\end{equation}
	Notice that
	\[
	\begin{aligned}
	&\mu_{k+1}(k+\alpha-1)^2-\mu_{k}(k+\alpha-2)^2\\
	=&\mu_{k+1}(k+\alpha-1)\left(k+\alpha-1-\frac{(k+\alpha-2)\ln^{\sigma}(k+\alpha-1)}{\ln^{\sigma}(k+\alpha-2)}\right)\\
	\leq&\mu_{k+1}(k+\alpha-1),
	\end{aligned}
	\]
	then
	\begin{equation}\label{prop2-2}
	\begin{aligned}
	&\mu_{k+1}(k+\alpha-1)^2W_{k+1}-\mu_{k}(k+\alpha-2)^2W_{k}\\
	=&\mu_{k+1}(k+\alpha-1)^2\left(W_{k+1}-W_{k}\right)
	+\left(\mu_{k+1}(k+\alpha-1)^2-\mu_{k}(k+\alpha-2)^2
	\right)W_{k}\\
	\leq&\mu_{k+1}(k+\alpha-1)^2\left(W_{k+1}-W_{k}\right)+\mu_{k+1}(k+\alpha-1)W_{k}.
	\end{aligned}
	\end{equation}
	
	For simplicity of notation, denote
	\begin{equation}\label{eq6}
	\alpha_k:=(k-1)^2\|x^{k}-x^{k-1}\|^2+2\gamma_{k}\mu_{k}(k+\alpha-2)^2W_{k}.
	\end{equation}
	Substituting (\ref{prop2-2}) into (\ref{prop2-1}), by $\gamma_{k+1}\leq\gamma_{k}$, and upon rearranging the terms, we deduce that
	\begin{equation}\label{eq5}
	\alpha_{k+1}-\alpha_k\leq2\gamma_{k+1}\mu_{k+1}(k+\alpha-1)W_{k}.
	\end{equation}
	Taking the positive part of the left-hand and thanks to (\ref{prop2-3}), we find
	$$\sum_{k=1}^{\infty}(\alpha_{k+1}-\alpha_k)_+<\infty.$$
	Since $\alpha_k\geq0$, by Lemma \ref{add2}, we infer that
	$$\lim_{k\rightarrow\infty}\alpha_k\quad \mbox{exists}.$$
	
	(iii). In view of $\alpha>3$, we observe that
	\[\begin{aligned}
	&(k+\alpha-1)^2\|x^{k+1}-x^{k}\|^2-(k-1)^2\|x^k-x^{k-1}\|^2\\
	\geq&(k+2)^2\|x^{k+1}-x^{k}\|^2-(k-1)^2\|x^k-x^{k-1}\|^2\\
	\geq&k^2\|x^{k+1}-x^{k}\|^2-(k-1)^2\|x^k-x^{k-1}\|^2+4k\|x^{k+1}-x^{k}\|^2,
	\end{aligned}\]
	combining which with (\ref{eq5}), then we obtain
	\[\begin{aligned}
	&k^2\|x^{k+1}-x^{k}\|^2-(k-1)^2\|x^k-x^{k-1}\|^2+4k\|x^{k+1}-x^{k}\|^2\\
	\leq&2\gamma_{k}\mu_{k}(k+\alpha-2)^2W_{k}-2\gamma_{k+1}\mu_{k+1}(k+\alpha-1)^2W_{k+1}
	+2\gamma_{k+1}\mu_{k+1}(k+\alpha-1)W_{k}.
	\end{aligned}\]
	Summing up the above inequality for $k=1,2,\ldots,K$, we obtain
	\begin{equation}\label{add1}
	\begin{aligned}
	&K^2\|x^{K+1}-x^{K}\|^2+4\sum_{k=1}^Kk\|x^{k+1}-x^{k}\|^2 \\
	\leq &2\gamma_{1}\mu_{1}(\alpha-1)^2W_{1}
	+2\sum_{k=1}^K\gamma_{k+1}\mu_{k+1}(k+\alpha-1)W_{k}.
	\end{aligned}
	\end{equation}
	Since $W_1\leq\mathcal{E}_1$,
	letting $K$ tend to infinity in the above inequality, by (\ref{prop2-3}) and (\ref{add1}), we have $\sum_{k=1}^{\infty}k\|x^{k+1}-x^{k}\|^2
	<\infty$.
\end{proof}
\subsection{Convergence rate for the objective values}\label{section3.3}
Thanks to the above analysis, we are now ready to give the global convergence rate of $f(x^k)$
to $\min_{\mathcal{X}}f$.
\begin{theorem}\label{theorem1}
	Let $\{x^k\}$ be the sequence generated by the SAPG algorithm. Then,
	\begin{equation}\label{thm1-3}
	\lim_{k\rightarrow\infty}(k+\alpha-2)\ln^{-\sigma}(k+\alpha-1)(f(x^k)-\min_{\mathcal{X}}f)=0
	\end{equation}
	and
	\begin{equation}\label{thm1-2}
	\liminf_{k\rightarrow\infty}(k+\alpha-2)\ln^{1-\sigma}(k+\alpha-2)(f(x^k)-\min_{\mathcal{X}}f)=0.
	\end{equation}
\end{theorem}
\begin{proof}
	Combining (i) and (iii) in Proposition \ref{prop2}, we have
	$$\sum_{k=1}^{\infty}\left[(k-1)\|x^{k}-x^{k-1}\|^2+2\gamma_{k}\mu_{k}(k+\alpha-2)W_k\right]<\infty,
	$$
	which implies
	\[
	\begin{aligned}
	&\sum_{k=1}^{\infty}\frac{\ln(k+\alpha-2)}{(k+\alpha-2)\ln(k+\alpha-2)}\left[(k-1)^2\|x^{k}-x^{k-1}\|^2+2\gamma_{k}\mu_{k}(k+\alpha-2)^2W_k\right]\\
	\leq&\sum_{k=1}^{\infty}\frac{1}{k+\alpha-2}\left[(k-1)(k+\alpha-2)\|x^{k}-x^{k-1}\|^2
	+2\gamma_{k}\mu_{k}(k+\alpha-2)^2W_k\right]\\
	=&\sum_{k=1}^{\infty}\left[(k-1)\|x^{k}-x^{k-1}\|^2
	+2\gamma_{k}\mu_{k}(k+\alpha-2)W_k\right]
	<\infty.
	\end{aligned}
	\]
	Observe that $\sum_{k=1}^{\infty}\frac{1}{(k+\alpha-1)\ln(k+\alpha-1)}=\infty$, then
	\begin{equation}\label{prop2-6}
	\liminf_{k\rightarrow\infty}\ln(k+\alpha-2)\left((k-1)^2\|x^{k}-x^{k-1}\|^2
	+2\gamma_{k}\mu_{k}(k+\alpha-2)^2W_k\right)=0.
	\end{equation}
	Combining this with Proposition \ref{prop2}-(ii), we obtain
	\begin{equation}
	\lim_{k\rightarrow\infty}(k-1)^2\|x^{k}-x^{k-1}\|^2+2\gamma_{k}\mu_{k}(k+\alpha-2)^2W_k=0,
	\end{equation}
	by $W_k\geq0$, which further implies
	\begin{equation}\label{prop2-5}
	\lim_{k\rightarrow\infty}(k-1)^2\|x^{k}-x^{k-1}\|^2=0\quad\mbox{and}\quad
	\lim_{k\rightarrow\infty}\gamma_{k}\mu_{k}(k+\alpha-2)^2W_k=0.
	\end{equation}
	Recalling the definition of $\mu_{k}$ in (\ref{algo2}), $\gamma_{k}\geq\underline{\gamma}>0$ and (\ref{prop1-15}), the second equation in (\ref{prop2-5}) implies
	(\ref{thm1-3}).
	Similarly, (\ref{prop2-6}) gives
	(\ref{thm1-2}).
\end{proof}

(\ref{thm1-3}) in Theorem \ref{theorem1} illustrates that for any $\sigma\in(\frac{1}{2},1]$ in the SAPG algorithm, it holds $f(x^k)-\min_{\mathcal{X}}f
=o(\ln^{\sigma}k/k)$.

\subsection{Sequential convergence} \label{section3.4}
In this subsection, we begin to analyze the convergence of the iterates generated by the SAPG algorithm

Opial's Lemma was first used to analyze the convergence of nonlinear contraction semigroups \cite{Bruck1975}. Here, we state the discrete version of it to prepare for analyzing the convergence of sequence.
\begin{lemma}\cite{opial}\label{opi}
	Let $S$ be a nonempty subset of $\mathbb{R}^n$ and $\{z_k\}$ be a sequence of $\mathbb{R}^n$. Assume that
	\begin{itemize}
		\item[{\rm(i)}] $\lim_{k\to\infty}\|z_k-z\|$ exists for every $z\in S$;
		
		\item[{\rm(i)}] every sequential limit point of sequence $\{z_k\}$ as $k\to\infty$ belongs to $S$.
	\end{itemize}
	Then, as $k\to\infty$, $\{z_k\}$ converges to a point in $S$.
\end{lemma}

To prove the sequential convergence, we also need recall the following inequality on nonnegative sequences, which will be used in the forthcoming sequential
convergence result.
\begin{lemma}\label{lemma3}\cite{Attouch-MP}
	Assume $\alpha\geq3$. Let $\{a_k\}$ and $\{\omega_k\}$ be two sequences of nonnegative numbers such that
	\[
	a_{k+1}\leq\frac{k-1}{k+\alpha-1}a_k+\omega_k\]
	for all $k\geq1$. If $\sum_{k=1}^{\infty}k\omega_k<\infty$, then $\sum_{k=1}^{\infty}a_k<\infty$.
\end{lemma}
\begin{theorem}\label{theorem2} (\textbf{Sequential convergence})
	Let $\{x^k\}$ be the sequence generated by the SAPG algorithm.
	Then, $\{x^k\}$ converges to a point belonging to $\arg\min_{\mathcal{X}}f$ as $k\rightarrow\infty$.
\end{theorem}
\begin{proof}
	To apply Lemma \ref{opi} to prove the convergence of $\{x^k\}$, we first prove that any cluster point of $\{x^k\}$
	belongs to $\arg\min_{\mathcal{X}}f$. Suppose $\bar{x}$ is a cluster point of $\{x^k\}$ with subsequence $\{x^{k_j}\}$, then by the continuity of $f$ and $\{x^k\}
	\subseteq\mathcal{X}$, we have
	\[f(\bar{x})=\lim_{j\rightarrow\infty}f(x^{k_j})=\min_{\mathcal{X}}f,
	\]
	which implies $\bar{x}\in\arg\min_{\mathcal{X}}f$.
	
	Next, if we can verify that for every $x^*\in \arg\min_{\mathcal{X}}f$, $\lim_{k\rightarrow\infty}\|x^k-x^*\|$ exists, then we conclude
	the sequential convergence of $\{x^k\}$ by Lemma \ref{opi}.
	
	Set
	\begin{equation}\label{eq9}
	h_k:=\|x^k-x^*\|^2
	\end{equation}
	with $x^*\in \arg\min_{\mathcal{X}}f$.
	Multiplying $2\gamma_{k+1}\mu_{k+1}$ on the both sides of (\ref{prop1-2}) and by $W_{k+1}\geq0$, we obtain
	\[
	2\langle y^k-x^{k+1},y^k-x^*\rangle-\|x^{k+1}-y^k\|^2+4\kappa\gamma_{k+1}\mu_{k+1}^2\geq0,
	\]
	which can be reformulated by
	\[
	\|y^k-x^*\|^2-\|x^{k+1}-x^*\|^2+4\kappa\gamma_{k+1}\mu_{k+1}^2\geq0.
	\]
	Then, recalling the definition of $y^k$, we get
	\[\begin{aligned}
	\|x^{k+1}-x^*\|^2\leq&\|y^k-x^*\|^2+4\kappa\gamma_{k+1}\mu_{k+1}^2\\
	=&\|x^k-x^*\|^2+\left(\frac{k-1}{k+\alpha-1}\right)^2\|x^k-x^{k-1}\|^2\\
	&+2\left(\frac{k-1}{k+\alpha-1}\right)\langle x^k-x^*,x^k-x^{k-1}\rangle+4\kappa\gamma_{k+1}\mu_{k+1}^2\\
	=&\|x^k-x^*\|^2+\left(\left(\frac{k-1}{k+\alpha-1}\right)^2+\left(\frac{k-1}{k+\alpha-1}\right)\right)
	\|x^k-x^{k-1}\|^2\\
	&+\left(\frac{k-1}{k+\alpha-1}\right)\left(
	\|x^k-x^*\|^2-\|x^{k-1}-x^*\|^2\right)+4\kappa\gamma_{k+1}\mu_{k+1}^2\\
	\leq&\|x^k-x^*\|^2+2\|x^k-x^{k-1}\|^2\\
	&+\left(\frac{k-1}{k+\alpha-1}\right)\left(
	\|x^k-x^*\|^2-\|x^{k-1}-x^*\|^2\right)+4\kappa\gamma_{k+1}\mu_{k+1}^2,
	\end{aligned}\]
	which can be reformulated by
	\begin{equation}\label{eq7}
	(h_{k+1}-h_k)_+\leq\left(\frac{k-1}{k+\alpha-1}\right)(h_k-h_{k-1})_++2\|x^k-x^{k-1}\|^2+4\kappa\gamma_{k+1}\mu_{k+1}^2.
	\end{equation}
	%
	%which together with (\ref{thm2-3}) gives
	%\[\begin{aligned}
	%\epsilon_{k+1}-\epsilon_k
	%=&(k+\alpha-1)(h_{k+1}-h_k)-(k-1)(h_k-h_{k-1})\\
	%\leq&2(k+\alpha-1)\|x^k-x^{k-1}\|^2+4\kappa(k+\alpha-1)\gamma_{k+1}\mu_{k+1}^2.
	%\end{aligned}
	%\]
	Recalling the definition of $\mu_{k+1}$ in (\ref{algo2}) and the non-increasing property of $\gamma_k$ proved in Lemma \ref{lemma1}-(i), we get
	\begin{equation}\label{eq8}
	\sum_{k=1}^{\infty} k\gamma_{k+1}\mu_{k+1}^2
	\leq\sum_{k=0}^{\infty}\gamma_0\frac{1}{(k+\alpha-1)\ln^{2\sigma}(k+\alpha-1)}<\infty,
	\end{equation}
	which uses $\sigma>\frac{1}{2}$ in the last inequality.
	
	Applying the above result in Lemma \ref{lemma3} with $a_k=(h_{k}-h_{k-1})_+$, and by Proposition \ref{prop2}-(iii), we have
	\begin{equation}\label{eq4}
	\sum_{k=1}^{\infty}(h_{k+1}-h_k)_+<\infty.
	\end{equation}
	As a consequence, we can easily deduce the convergence of $h_k$ as $k\rightarrow\infty$ by the nonnegativity of $\{h_k\}$ and Lemma \ref{add2}. This completes the proof.
\end{proof}
\begin{remark}
	To obtain (\ref{eq4}), one key point is the additivity of
	$\sum_{k=1}^{\infty} k\mu_{k+1}^2$, which is guaranteed by the definition of $\mu_{k+1}$ in (\ref{algo2}) with $\sigma>\frac{1}{2}$.
\end{remark}

\subsection{Stability}
In this section, we consider an inexact version of the SAPG algorithm in Algorithm \ref{algorithm1}. The inexact version comes from the error on the computation of $\nabla\tilde{c}(y^k,\mu_{k+1})$ and we will give a tolerance estimate for the error sequence to guarantee the all convergence properties of SAPG algorithm in Theorem \ref{theorem1} and Theorem \ref{theorem2}. Here, since we allow some errors on the calculation of $\nabla\tilde{c}$, we need the prior knowledge on the Lipschitz constant of $\nabla\tilde{c}(x,\mu)$, which means that the constant $L$ in Definition \ref{defn1}-(vi) is known in advance,
and we fix the parameter $\gamma_k:=L^{-1}$ in Algorithm \ref{algorithm1}. The inexact version of Algorithm \ref{algorithm1} is shown in Algorithm \ref{algorithm2} and named by inexact smoothing accelerated proximal gradient (ISAPG) algorithm, where $\epsilon_k\in\mathbb{R}^n$ is an unknown error.
\begin{algorithm}[!h]
	\caption{Inexact Smoothing Accelerated Proximal Gradient (ISAPG) Algorithm }\label{algorithm2}
	\begin{itemize}
		\item[\textbf{Input:}]
		Take initial point $x^{-1}=x^0\in\mathcal{X}$ and $\mu_0\in\mathbb{R}_{++}$. Choose parameters $\alpha>3$ and $\sigma\in(\frac{1}{2}, 1]$.
		
		Set $k=0$.
		\item[\textbf{Step 1:}] Let
		\begin{equation}\label{i-algo1}
		y^k=x^k+\frac{k-1}{k+\alpha-1}(x^k-x^{k-1}) ,
		\end{equation}
		\begin{equation}\label{i-algo2}
		\mu_{k+1}=\frac{\mu_0}{(k+\alpha-1)\ln^{\sigma}(k+\alpha-1)}.
		\end{equation}
		\item[\textbf{Step 2:}] Compute
		\begin{equation}\label{i-algo3}
		x^{k+1}={\rm prox}_{L^{-1}{\mu_{k+1}}g}(y^k-L^{-1}\mu_{k+1}(\nabla\tilde{c}(y^k,\mu_{k+1})+{\epsilon_k})). \end{equation}
	\end{itemize}
\end{algorithm}

Notice that the Step 1 in Algorithm \ref{algorithm1} and Algorithm \ref{algorithm2}
are same. Denote
$$Q_k(x,y,\mu):=\tilde{c}(y,\mu)+\langle\nabla\tilde{c}(y,\mu)+\epsilon_k,x-y\rangle+\frac{1}{2}L\mu^{-1}\|x-y\|^2+g(x).$$
Then, $x^{k+1}$ in \eqref{i-algo3} is the solution of
$$\min_{x\in\mathcal{X}}Q_k(x,y^k,\mu_{k+1}).$$
This means that
$$y^k-L^{-1}\mu_{k+1}\nabla\tilde{c}(y^k,\mu_{k+1})\in
x^{k+1}+L^{-1}\mu_{k+1}(\partial g(x^{k+1})+\epsilon_k) + N_{\mathcal{X}}(x^{k+1}),$$
where $N_{\mathcal{X}}(x^{k+1})$ is the normal cone to $\mathcal{X}$ at $x^{k+1}$.
Due to the strong convexity of $Q_k$ to $x$ for fixed $y$ and $\mu$, we also have
\begin{equation}
Q_k(x,y^k,\mu_{k+1})\geq Q_k(x^{k+1},y^k,\mu_{k+1})+\frac{L}{2}\mu_{k+1}^{-1}\|x-x^{k+1}\|^2,\quad \forall x\in\mathcal{X}.
\end{equation}
The analysis method of the ISAPG algorithm is similar to the unperturbed case of SAPG algorithm. Hence, we state the main results, sketch the proofs to avoid repeating similar arguments, and underline the parts where additional techniques are required. Next, we first recall the discrete version of Gronwall-Bellman lemma.
\begin{proposition}\cite{Attouch-MP}\label{Attouch-lemma}
	Let $\{a_k\}$ be a sequence of nonnegative numbers satisfying
	\begin{equation*}
	a_k^2\leq c^2+\sum_{j=1}^k\beta_j a_j
	\end{equation*}
	for all $k\in\mathbb{N}$, where $c\geq 0$ and $\{\beta_j\}$ is a summable sequence of nonnegative numbers. Then, $a_k\leq c+\sum_{j=1}^{\infty}\beta_j$ for all $k\in\mathbb{N}$.
\end{proposition}
\begin{theorem}\label{stathe}
	Let $\{x^k\}$ be the sequence generated by the ISAPG algorithm. If the errors satisfy $\sum_{k=1}^{\infty}\mu_{k+1}(k+\alpha-1)\|\epsilon_k\|<\infty$, then
	\begin{itemize}
		\item[{\rm(i)}]$\sum_{k=1}^{\infty}\mu_{k}(k+\alpha-2)W_{k}<\infty$,
		$\sum_{k=1}^{\infty}(k-1)\|x^{k}-x^{k-1}\|^2<\infty$;
		\item[{\rm(ii)}] $\lim_{k\rightarrow\infty}\left[(k-1)^2\|x^{k}-x^{k-1}\|^2+2L^{-1}\mu_{k}(k+\alpha-2)^2W_{k}\right]$ exists;
		\item[{\rm(iii)}] the conclusions in Theorem \ref{theorem1} and Theorem \ref{theorem2} hold.
	\end{itemize}
\end{theorem}
\begin{proof}
	(i). Since \eqref{algo4} always holds when $\hat{\gamma}_{k+1}=L^{-1}$, similar to the proof idea of Lemma \ref{lemma1} and Lemma \ref{lemma2}, it gives $x^k\in\mathcal{X}$, $\forall k$, $\lim_{k\rightarrow\infty}\mu_k=0$ and for all $x\in\mathcal{X}$,
	\begin{equation}\label{theorem3-1}
	\begin{aligned}
		&\tilde{f}(x^{k+1},\mu_{k+1}) \\
		\leq &\tilde{f}(x,\mu_{k+1})
	+\frac{1}{2}L\mu_{k+1}^{-1}\|x-y^k\|^2
	-\frac{1}{2}L\mu_{k+1}^{-1}\|x-x^{k+1}\|^2+\langle\epsilon_k,x-x^{k+1}\rangle.
	\end{aligned}
	\end{equation}
	In what follows, the key idea for proving the results in this theorem relies on the following energy function:
	$$\bar{\mathcal{E}}_k={\mathcal{E}}_k+2L^{-1}\sum_{j=k}^{\infty}\mu_{j+1}(j+\alpha-1)
	\langle\epsilon_j, x^*- u^{j+1}\rangle,$$
	where $x^*$, ${\mathcal{E}}_k$ and $u^k$ are defined as in \eqref{prop1-10} and \eqref{prop1-11}. Following the proof ideas in Proposition \ref{prop1}, we obtain
	$${\mathcal{E}}_{k+1}+\frac{2L^{-1}(\alpha-3)\mu_{k+1}}{\alpha-1}(k+\alpha-1)W_k
	\leq{\mathcal{E}}_k+2L^{-1}\mu_{k+1}(k+\alpha-1)\langle\epsilon_k,x^*-u^{k+1}\rangle,$$
	which implies
	\begin{equation}\label{theorem3-2}
	\bar{\mathcal{E}}_{k+1}+\frac{2L^{-1}(\alpha-3)\mu_{k+1}}{\alpha-1}(k+\alpha-1)W_k
	\leq\bar{\mathcal{E}}_k.
	\end{equation}
	Then, $\{\bar{\mathcal{E}}_k\}$ is non-increasing and $\bar{\mathcal{E}}_k\leq\bar{\mathcal{E}}_0$, for all $k\geq0$.
	By virtue of $\bar{\mathcal{E}}_k\leq\bar{\mathcal{E}}_0$ and $W_k\geq0$, we obtain
	$$\begin{aligned}
	\|u^k-x^*\|^2\leq&\frac{1}{\alpha-1}\mathcal{E}_0+\frac{2L^{-1}}{\alpha-1}\sum_{j=0}^{k-1}\mu_{j+1}(j+\alpha-1)\langle\epsilon_j,x^*-u^{j+1}\rangle\\
	\leq&\frac{1}{\alpha-1}\mathcal{E}_0+\frac{2L^{-1}}{\alpha-1}\sum_{j=1}^{k}\mu_{j}(j+\alpha-2)\|\epsilon_{j-1}\|\|u^{j}-x^*\|.
	\end{aligned}$$
	Applying Proposition \ref{Attouch-lemma} with $a_k=\|u^k-x^*\|$ and $\beta_j=\mu_{j}(j+\alpha-2)\|\epsilon_{j-1}\|$ to the above inequality, and by the condition on $\epsilon_k$ given in this theorem, we deduce that
	\begin{equation}\label{theorem3-7}
	\|u^k-x^*\|\leq\Gamma:=\sqrt{\frac{1}{\alpha-1}\mathcal{E}_0}+\frac{2L^{-1}}{\alpha-1}
	\sum_{j=1}^{\infty}\mu_j(j+\alpha-2)\|\epsilon_{j-1}\|,\quad\forall k\geq0.
	\end{equation}
	Then, the above bound with ${\mathcal{E}}_k\geq0$ gives
	$$\bar{\mathcal{E}}_k\geq-2L^{-1}\Gamma\sum_{j=k}^{\infty}\mu_{j+1}(j+\alpha-1)
	\|\epsilon_j\|>-\infty.$$
	Thus, $\bar{\mathcal{E}}_k$ is bounded from below and \eqref{theorem3-2} gives
	\begin{equation}\label{theorem3-3}
	\sum_{k=1}^{\infty}\mu_{k+1}(k+\alpha-1)W_k<\infty\quad\mbox{and}\quad
	\lim_{k\rightarrow\infty}\bar{\mathcal{E}}_k\quad \mbox{exists.}
	\end{equation}
By a same analysis of \eqref{pro_add}, it holds that
\begin{equation*}
\mu_k(k+\alpha-2)\frac{\ln^{\sigma}(\alpha-1)}{\ln^{\sigma}\alpha}\leq\mu_{k+1}(k+\alpha-1),
\end{equation*}
	this together with \eqref{theorem3-3}, we obtain
	$\sum_{k=1}^{\infty}\mu_{k}(k+\alpha-2)W_k<\infty.$

	Performing a similar analysis of \eqref{prop2-4}-\eqref{prop2-2} but not loosing $k+\alpha-1$ to $k$, we have
	\begin{equation}\label{theorem3-5}
	\begin{aligned}
	&(k+\alpha-1)^2\|x^{k+1}-x^k\|^2-(k-1)^2\|x^k-x^{k-1}\|^2\\
	&\quad +2L^{-1}\mu_{k+1}(k+\alpha-1)^2W_{k+1}-2L^{-1}\mu_k(k+\alpha-2)^2W_k\\
	\leq&2L^{-1}\mu_{k+1}(k+\alpha-1)W_k+2L^{-1}\mu_{k+1}(k+\alpha-1)^2\|\epsilon_k\|\|x^k-x^{k+1}\|.
	\end{aligned}
	\end{equation}
	Summing up the above inequality for $k=1,\ldots,K$ and thanks to $W_k\geq0$, we obtain
	\begin{equation}\label{theorem3-4}
	\begin{aligned}
	&(K+\alpha-1)^2\|x^{K+1}-x^K\|^2+\sum_{k=1}^{K-1}((k+\alpha-1)^2-k^2)\|x^{k+1}-x^k\|^2\\
	\leq& C^2+2L^{-1}\sum_{k=1}^K\mu_{k+1}(k+\alpha-1)^2\|\epsilon_k\|\|x^k-x^{k+1}\|,
	\end{aligned}
	\end{equation}
	where $C:=\sqrt{2L^{-1}(\mu_1(\alpha-1)^2 W_1+\sum_{k=1}^{\infty}\mu_{k+1}(k+\alpha-1)W_k)}$ is a finite number due to \eqref{theorem3-3}. Then, we deduce that, for $\forall K\geq1$,
	\begin{equation}\label{theorem3-6}
	(K+\alpha-1)^2\|x^{K+1}-x^K\|^2\leq C^2+2L^{-1}\sum_{k=1}^K\mu_{k+1}(k+\alpha-1)^2\|\epsilon_k\|\|x^k-x^{k+1}\|.
	\end{equation}
	Applying Proposition \ref{Attouch-lemma} with $a_k=(k+\alpha-1)\|x^k-x^{k+1}\|$ and $\beta_j=2L^{-1}\mu_{j+1}(j+\alpha-1)\|\epsilon_j\|$ into the above inequality gives
	\begin{equation}\label{theorem3-8}
	\begin{aligned}
	&(k+\alpha-1)\|x^{k+1}-x^k\| \\
	\leq &M:=C+2L^{-1}\sum_{j=1}^{\infty}\mu_{j+1}(j+\alpha-1)\|\epsilon_j\|<\infty,\quad\forall k\geq1.
	\end{aligned}
	\end{equation}
	Substituting the above bound to \eqref{theorem3-4} and recalling $\alpha\geq3$, we have
	$$4\sum_{k=1}^{\infty}k\|x^{k+1}-x^k\|^2\leq C^2+2L^{-1}M\sum_{k=1}^{\infty}\mu_{k+1}(k+\alpha-1)\|\epsilon_k\|<\infty,$$
	which implies $\sum_{k=1}^{\infty}(k-1)\|x^{k}-x^{k-1}\|^2<\infty$.
	
	(ii). Define $\alpha_k$ as in \eqref{eq6}. Injecting $k+\alpha-1>k$ and \eqref{theorem3-6} to \eqref{theorem3-5}, we obtain
	$$\alpha_{k+1}-\alpha_k\leq2L^{-1}\mu_{k+1}(k+\alpha-1)W_k+2L^{-1}M\mu_{k+1}(k+\alpha-1)\|\epsilon_k\|.$$
	In view of \eqref{theorem3-3} and $\sum_{k=1}^{\infty}\mu_{k+1}(k+\alpha-1)\|\epsilon_k\|<\infty$, taking the positive part of the left-hand of the above inequality and by $\alpha_k\geq0$, we obtain that $\lim_{k\rightarrow\infty}\alpha_k$ exists by Lemma \ref{add2}, which is just the result in (ii).
	
	(iii). Based on the previous results, we can proceed as in the proof of Theorem \ref{theorem1} to obtain the results in Theorem \ref{theorem1} for the ISAPG algorithm.
	
	To obtain the sequence convergence of $\{x^k\}$,
	we follow the proof of Theorem \ref{theorem2} and only need to prove the existence of
	$\lim_{k\rightarrow\infty}\|x^k-x^*\|$ for all $x^*\in \arg\min_{\mathcal{X}}f$.
	Similarly, define $h_k$ as in \eqref{eq9} and we can obtain
	\begin{equation*}
\begin{aligned}
(h_{k+1}-h_k)_+\leq &\left(\frac{k-1}{k+\alpha-1}\right)(h_k-h_{k-1})_++2\|x^k-x^{k-1}\|^2 \\
&+4\kappa L^{-1}\mu_{k+1}^2
+2L^{-1}\mu_{k+1}\|\epsilon_k\|\|x^{k+1}-x^*\|.
\end{aligned}
	\end{equation*}

	Recalling the definition of $u^k$ in \eqref{prop1-11}, we get
	$$\|u^k-x^*\|=\left\|x^k+\left(\frac{k-1}{\alpha-1}\right)(x^k-x^{k-1})-x^*\right\|
	\geq\|x^k-x^*\|-\frac{k-1}{\alpha-1}\|x^k-x^{k-1}\|,$$
	which implies that the sequence $\{\|x^k-x^*\|\}$ is bounded by \eqref{theorem3-7} and \eqref{theorem3-8}.
	Therefore, also using Lemma \ref{lemma3} with $a_k=(h_{k}-h_{k-1})_+$, we have
	\begin{equation}\label{eq41}
	\sum_{k=1}^{\infty}(h_{k+1}-h_k)_+<\infty,
	\end{equation}
	which uses (i), \eqref{eq8} and $\sum_{k=1}^{\infty}\mu_{k+1}(k+\alpha-1)\|\epsilon_k\|<\infty$.
	Therefore, we complete the proof by the nonnegativity of $\{h_k\}$.
\end{proof}
\begin{remark}
	When $c$ in problem \eqref{ob} is a Lipschitz continuously differentiable convex function, Attouch and Peypouquet in \cite{Attouch-Peypouquet} studied the stability of the accelerated forward-backward algorithm with extrapolation $\frac{k-1}{k+\alpha-1}$, which enjoys fast convergence rate $o(k^{-2})$ on the objective function values and sequential convergence under $\sum_{k=1}^\infty k\|\epsilon_k\|<\infty$. In this paper, we only require $c$ to be a continuous convex function. Although the convergence rate of the objective function values is $o(\frac{\ln^{\sigma} k}{k})$ with any $\sigma\in(\frac{1}{2},1]$, the exact condition on the errors in Theorem \ref{stathe} is $\sum_{k=1}^\infty \ln^{-\sigma}(k+\alpha-1)\|\epsilon_k\|<\infty$,
	which is much weaker than it in \cite{Attouch-Peypouquet}.
\end{remark}

\section{Numerical experiments}\label{section4}
In this section, we present numerical results to show the good performance of the SAPG algprithm for solving \eqref{ob}. The numerical experiments are performed in Python 3.7.0 on a 64-bit Lenovo PC with an Intel(R) Core(TM) i7-10710U CPU @1.10GHz 1.61GHz and 16GB RAM.
We test two popular optimization models in practical applications: linear regression problem and censored regression problem. In this paper, we call the SAPG algorithm without extrapolation the smoothing proximal gradient (SPG) algorithm. In order to illustrate the acceleration effect of the SAPG algorithm, we make a comparison between the convergence rate of the SAPG algorithm and the SPG algorithm in each example.
The objective functions in the following examples have the form of \eqref{ob}, and the smoothing functions of the loss functions used in the following numerical experiments are defined in Example 3.1 of reference \cite{Bian-Chen-SINMU}.

In view of the convexity of the objective function in \eqref{ob}, $x^*$ is a minimizer of \eqref{ob} if and only if $x^*$ satisfies
$$x^*\in P_{\mathcal{X}}(x^*-\zeta\partial{f}(x^*))$$
with a number $\zeta>0$. Thus,
our stopping criterion is set as
\begin{equation}\label{stop1}
\mbox{number of iterations}>\mbox{Maxiter}
\end{equation}
or
\begin{equation}\label{stop2}
\|x^k-P_{\mathcal{X}}(x^k-\zeta\nabla\tilde{f}(x^k,\mu_k))\|_{\infty}\leq\epsilon\ \mbox{and}\ \mu_k\leq\epsilon,
\end{equation}
%\begin{subequations}\label{stop}
%	\begin{align}
%	&\mbox{number of iterations}>\mbox{Maxiter}\label{stop1}\\
%	\mbox{or}\quad &\|x^k-P_{\mathcal{X}}(x^k-\zeta\nabla\tilde{f}(x^k,\mu_k))\|_{\infty}\leq\epsilon\ \mbox{and}\ \mu_k\leq\epsilon,\label{stop2}
%	\end{align}
%\end{subequations}
where "Maxiter" is the given positive integer to indicate the number of iterations allowed,
and $\zeta\in\mathbb{R}_{++}$ is a given positive parameter. If $\epsilon=0$, \eqref{stop2} implies that $x^k$ is a minimizer of problem \eqref{ob}.
Namely, we stop the algorithm by \eqref{stop1} or \eqref{stop2}, i.e. the number of iterations exceeds Maxiter or the iterate $x^k$ is an $\epsilon$ minimizer of the problem.
The CPU time reported here in seconds does not include the time for data initialization. To guarantee the fairness of the comparison, we use same parameters and initial point in two algorithms.
The values of parameters in the numerical experiments are chosen as follows:
\begin{equation*}
\mbox{Maxiter}=15000,\ \epsilon =10^{-3},\ \zeta=3\times 10^{-3},
%\ \mu_0=0.8,\ \gamma_0=1,\ \eta=\frac{1}{2},\ \alpha=4,\ \sigma=\frac{3}{4}.
\end{equation*}
and
\begin{equation*}
%\mbox{Maxiter}=1.5\times10^{4},\ \epsilon =10^{-3},\ \zeta=3\times 10^{-3},
\mu_0=0.8,\ \gamma_0=1,\ \eta=\frac{1}{2},\ \alpha=4,\ \sigma=\frac{3}{4}.
\end{equation*}

For simplicity, we use Time, Iter and Spar to represent the CPU time in seconds, the number of iterations and the sparsity level of the true solution for generating data. Moreover, we set $s=\mbox{Spar}\ast n$, which means that there are at most $s$ elements of the generated solution are nonzero. The initial point is chosen as $x^0=0.1\ast \mathbf{1}_n$, where $\mathbf{1}_n$ denotes the vector of all ones.

Moreover, we use $f_{\min}$ to denote the minimum of the two objective values at the stopped iterates obtained from the SAPG algorithm and the SPG algorithm in the following examples.

\begin{example}\label{exa1}
	We consider the following $\ell_1$ penalized linear regression problem with $\ell_1$ loss function:
	
	\begin{equation}\label{e1}
	\min_{\mathbf{0}\leq x\leq\mathbf{1}} \|Ax-b\|_1 + 0.01\|x\|_1,
	\end{equation}
	where $A\in\mathbb{R}^{m\times n}$ with $m<n$, $b\in\mathbb{R}^m$.
	For a given group of $(m,n,\mbox{Spar})$, the codes of generating random data in \eqref{e1} are
	\begin{center}
		\tt{\text{$B=\mbox{np.random.randn}(m,n);\ A=\mbox{orth}(B\mbox{.T})\mbox{.T};\ s=\mbox{Spar}*n;$}}\\
		\tt{\text{$x^{\star}=\mbox{np.random.uniform}(0,1,(n,1));\ x^\star[:n-s]=0;\ \mbox{np.random.shuffle}(x^\star);$}}\\
		\tt{\text{ $\bar{b}=A\mbox{.dot}(x^\star);\ b=\bar{b}+0.01\ast \mbox{np.random.rand}(\bar{b}\mbox{.shape}[0], \bar{b}\mbox{.shape}[1]).$}}
	\end{center}
		
In the numerical experiments, the smoothing function of the $\ell_1$ loss function is chosen as below \cite{Chen-MP}:
	\begin{equation*}
	\begin{aligned}
	\tilde{f}(x,\mu)=\sum_{i=1}^m\tilde{\theta}(A_i x-b_i, \mu)
	\end{aligned}
	\quad\mbox{with}\quad
	\tilde{\theta}(z,\mu)=\left\{
	\begin{aligned}
	&|z|  &\mbox{if}~ |z|>\mu, \\
	&\frac{z^2}{2\mu}+\frac{\mu}{2}  &\mbox{if}~ |z|\leq \mu.
	\end{aligned}
	\right.
	\end{equation*}	
	
	The number of iterations and CPU time are taken into consideration to illustrate the performance of the SAPG algorithm and the SPG algorithm. We compare the two algorithms by setting different dimensions of $A$ and the sparsity levels of $x^\star$.
	By running 50 independent trials for each $(m, n, \mbox{Spar})$,
	the average values of the numerical results for finding an $\epsilon$ minimizer of problem \eqref{e1} defined by \eqref{stop2} are recorded in Table \ref{tab1}.
	We see that in Table \ref{tab1} the average numbers of iterations and CPU time cost by the SAPG algorithm are smaller than that spent by the SPG algorithm for each case, which means the SAPG algorithm performs better than the SPG algorithm for problem \eqref{e1}.
	
	\begin{table}[htbp]
		{\footnotesize
			\caption{The average computational cost for Example 4.1 with different $m$, $n$ and $\mbox{Spar}$.}\label{tab1}
			\begin{center}
				\begin{tabular}{|c|c|c|c|c|c|c|c|c|c|} \hline
					
					\multicolumn{2}{|c|}{Methods} & SAPG & SPG& SAPG & SPG & SAPG & SPG & SAPG & SPG \\   \hline
					
					\multicolumn{2}{|c|}{\diagbox{Spar}{Costs}{$(m,n)$}}	& \multicolumn{2}{c|}{$(150,300)$} &  \multicolumn{2}{c|}{$(300,600)$} & \multicolumn{2}{c|}{$(450,900)$}& \multicolumn{2}{c|}{$(600,1200)$}\\ \hline
					
					\multirow{2}{*}{$20 \%$}& Iter & $\mathbf{223}$ & 251 & $\mathbf{223}$ & 247 & $\mathbf{223}$& 243& $\mathbf{223}$& 245 \\ %\hline
					\cline{2-10}
					&   Time  &  $\mathbf{0.0592}$ & 0.0647 & $\mathbf{0.1027}$ & 0.1152 & $\mathbf{0.1664}$& 0.1860& $\mathbf{0.2349}$& 0.2660 \\ \hline
					
					\multirow{2}{*}{$30 \%$}    & Iter & $\mathbf{223}$ & 317 & $\mathbf{223}$ & 413 & $\mathbf{223}$& 492& $\mathbf{223}$& 480 \\ %\hline
					\cline{2-10}
					&Time  & $\mathbf{0.0585}$ & 0.0816 & $\mathbf{0.1039}$ & 0.1940 & $\mathbf{0.1667}$& 0.3751 & $\mathbf{0.2369}$& 0.5377\\ \hline
					
					\multirow{2}{*}{$40 \%$}    & Iter & $\mathbf{223}$ & 777 & $\mathbf{223}$ & 875 & $\mathbf{223}$& 897& $\mathbf{223}$& 886 \\ %\hline
					\cline{2-10}
					& Time & $\mathbf{0.0584}$ & 0.1979 & $\mathbf{0.1028}$ & 0.4070 & $\mathbf{0.1699}$& 0.7172& $\mathbf{0.2344}$& 1.0081 \\ \hline
					
					\multirow{2}{*}{$50 \%$}    & Iter & $\mathbf{223}$ & 911 & $\mathbf{223}$ & 1343 & $\mathbf{223}$& 1622 & $\mathbf{223}$& 1800\\ %\hline
					\cline{2-10}
					& Time & $\mathbf{0.0569}$ & 0.2299 & $\mathbf{0.1035}$ & 0.6282 & $\mathbf{0.1712}$& 1.3444 & $\mathbf{0.2523}$& 2.3463\\ \hline
				\end{tabular}
			\end{center}
		}
	\end{table}
	
	For $(m,n)=(300,600)$ and $(m,n)=(600,1200)$ with two different sparsity levels $\mbox{Spar}=30 \%$ and $\mbox{Spar}=50 \%$, the corresponding results, that are $f(x^k)-f_{\min}$ versus the number of iterations $k$, are plotted in Figure \ref{fig1} and Figure \ref{fig2}, respectively. Seeing Figure \ref{fig1} and Figure \ref{fig2}, one can clearly find that the SAPG algorithm can find a more accurate solution with much fewer iterations than the SPG algorithm.
	
	\begin{figure}[htbp]
		\centerline{
			\subfigure[$(m,n,\mbox{Spar})=(300,600,30\%)$.]{\label{fig1.a}\includegraphics[width=0.55\textwidth]{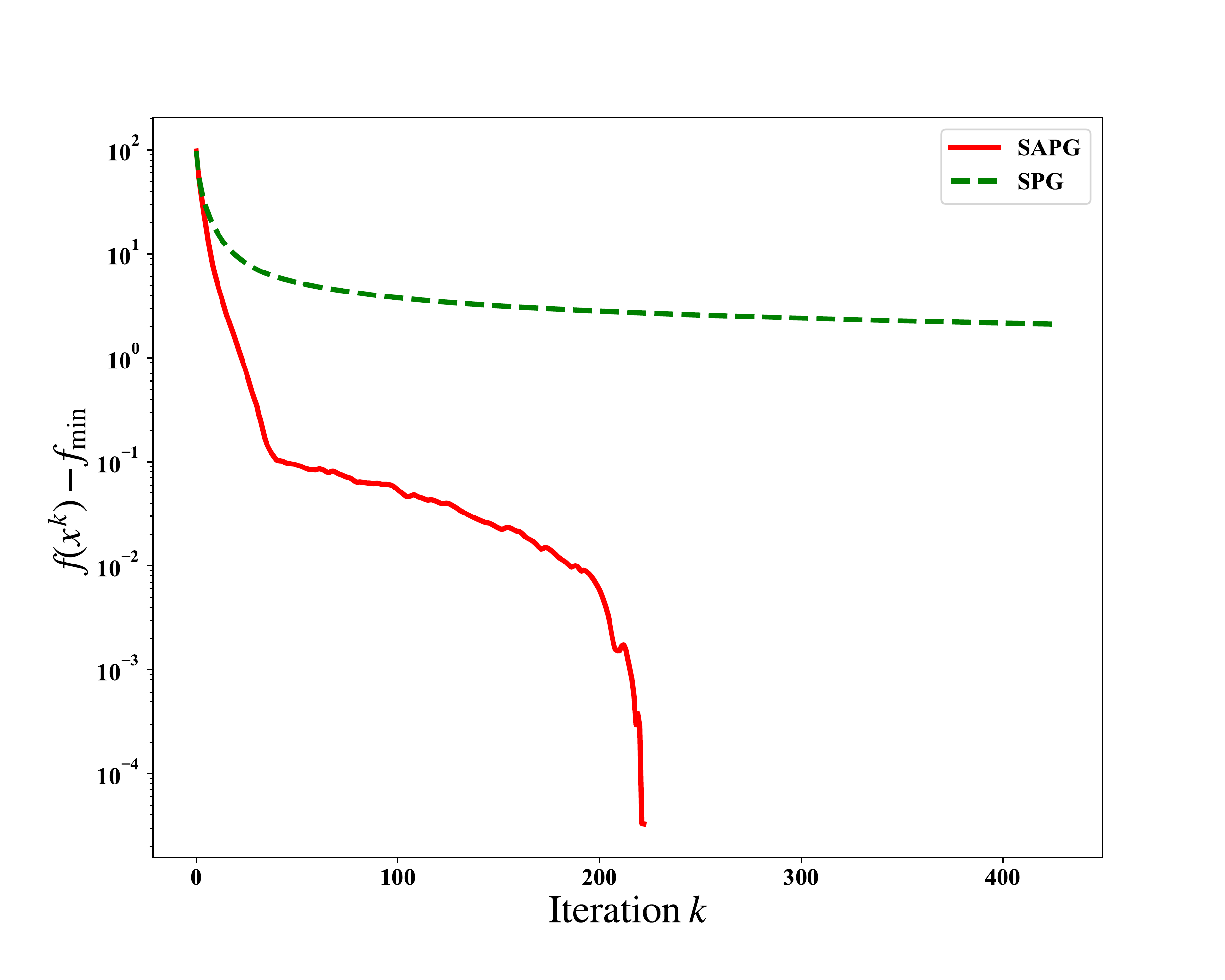}}
			\subfigure[$(m,n,\mbox{Spar})=(300,600,50\%)$.]{\label{fig1.b}\includegraphics[width=0.55\textwidth]{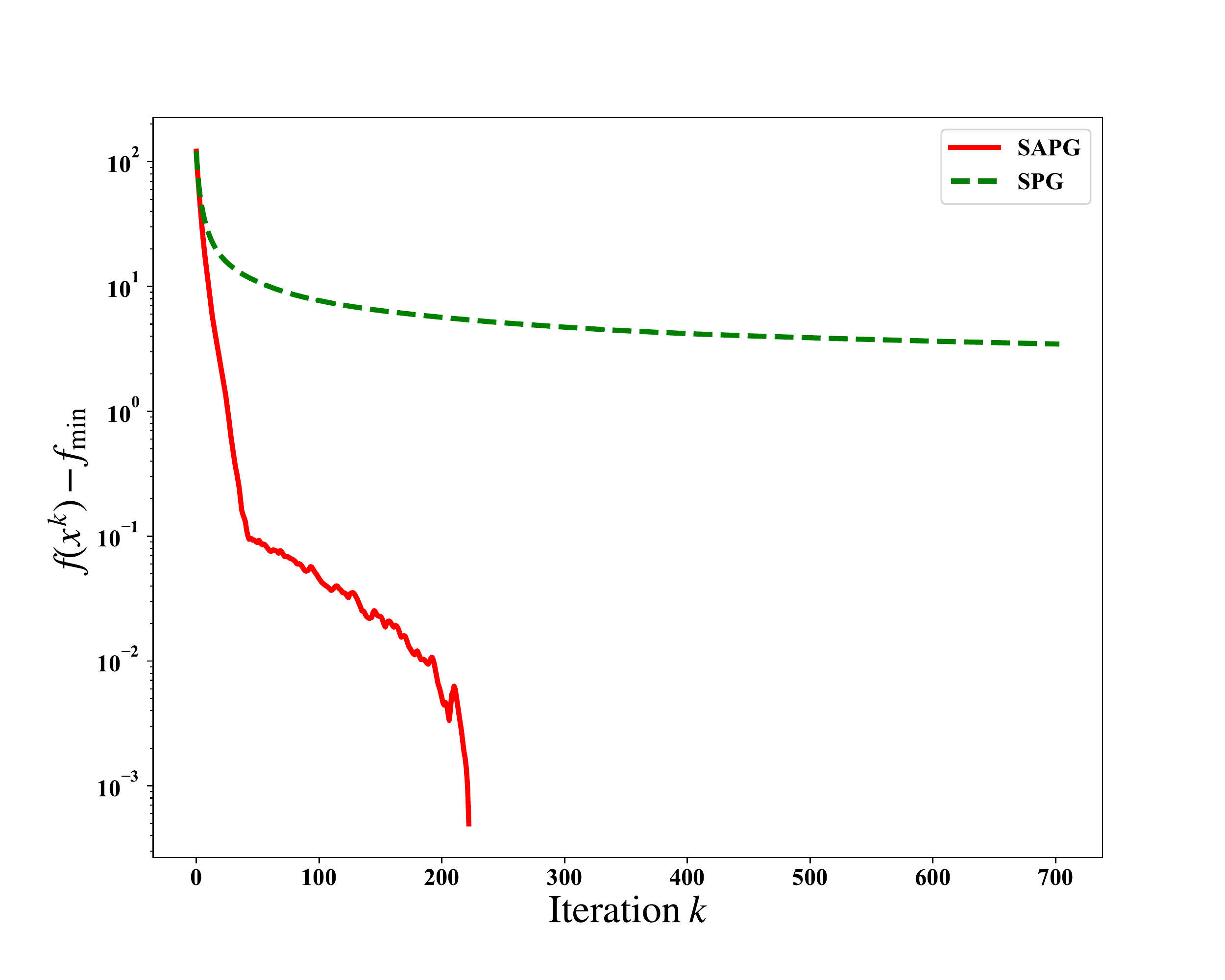}}
		}
		\caption{Convergence of $\{f(x^k)-f_{\min}\}$ for Example \ref{exa1} with $m=300$, $n=600$ and different sparsity levels.}
		\label{fig1}       % Give a unique label
	\end{figure}
	
	\begin{figure}[htbp]
		\centerline{
			\subfigure[$(m,n,\mbox{Spar})=(600,1200,30\%)$.]{\label{fig1.a}\includegraphics[width=0.55\textwidth]{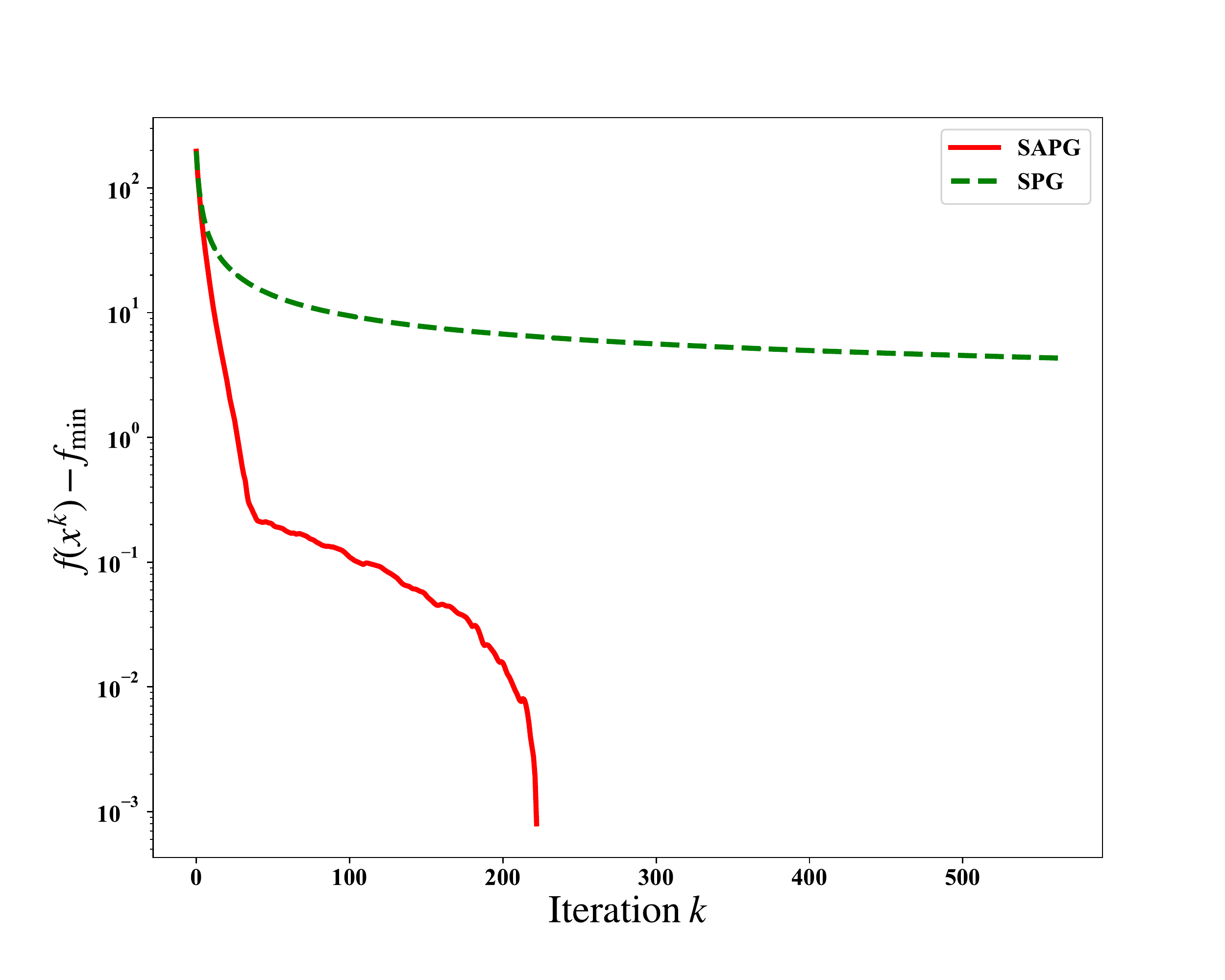}}
			\subfigure[$(m,n,\mbox{Spar})=(600,1200,50\%)$.]{\label{fig1.b}\includegraphics[width=0.55\textwidth]{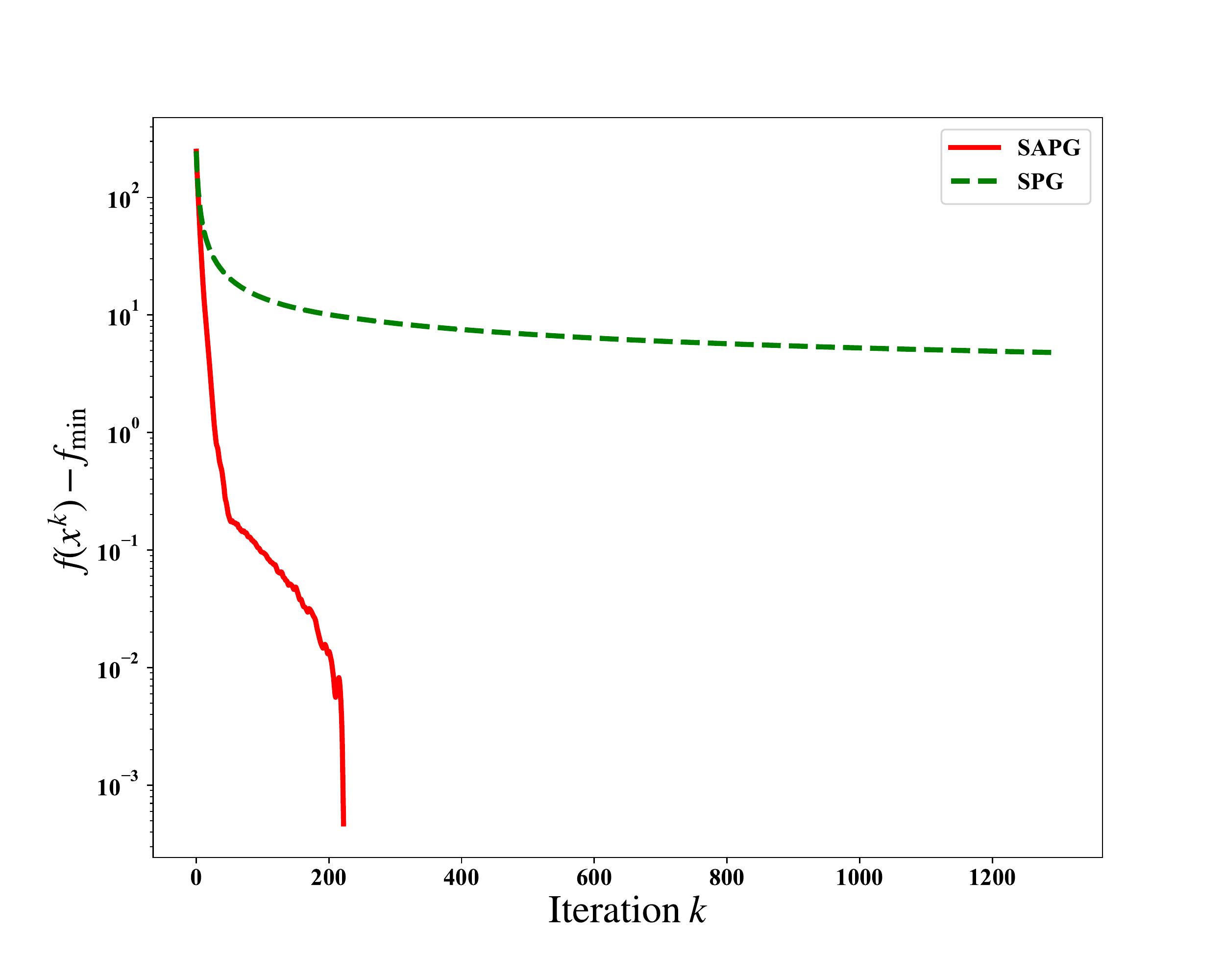}}
		}
		\caption{Convergence of $\{f(x^k)-f_{\min}\}$ for Example \ref{exa1} with $m=600$, $n=1200$ and different sparsity levels.}
		\label{fig2}       % Give a unique label
	\end{figure}

\end{example}

\begin{example}\label{exa2}
	We consider the following $\ell_1$ penalized censored regression problem:
	\begin{equation}\label{e2}
	\min_{\mathbf{0}\leq x\leq\mathbf{1}} \|\max\{Ax,0\}-b\|_1 + 0.01\|x\|_1,
	\end{equation}
	where the loss function is defined in \eqref{max-l1} with $q=1$. For a given group of $(m,n,\mbox{Spar})$, the data in this example is generated as follows:
	\begin{center}
		\tt{\text{$B=\mbox{np.random.randn}(m,n);\ A=\mbox{orth}(B\mbox{.T})\mbox{.T};\ s=\mbox{Spar}*n;$}}\\
		\tt{\text{$x^{\star}=\mbox{np.random.uniform}(0,1,(n,1));\ x^\star[:n-s]=0;\ \mbox{np.random.shuffle}(x^\star);$}}\\
		\tt{\text{$\bar{b}=A\mbox{.dot}(x^\star);\ \mbox{per}=0.01\ast \mbox{np.random.rand}(\bar{b}\mbox{.shape}[0], \bar{b}\mbox{.shape}[1]);$}}\\
		\tt{\text{ $ b=\mbox{np.maximum}(\bar{b}+\mbox{per}, \mbox{np.zeros}(\bar{b}\mbox{.shape})).$}}
	\end{center}
	
	 A smoothing function of the loss function in \eqref{e2} satisfying Definition \ref{defn1} can be defined by \cite{Chen-MP}
	\begin{equation*}
	\begin{aligned}
	\tilde{f}(x,\mu)=\sum_{i=1}^m\tilde{\theta}(\tilde{\phi}(A_i x, \mu)-b_i, \mu)
	\end{aligned}
	\quad\mbox{with}\quad
	\tilde{\phi}(z,\mu)=\left\{
	\begin{aligned}
	&\max\{z,0\} &\mbox{if}~ |z|>\mu, \\
	&\frac{(z+\mu)^2}{4\mu} &\mbox{if}~ |z|\leq \mu.
	\end{aligned}
	\right.
	\end{equation*}
	
	For each fixed $(m,n,\mbox{Spar})$, we also randomly and independently generate 50 sets of data. In Table \ref{tab2}, we report the average values of iterations and CPU time for these 50 independent tests. We can see that the SAPG algorithm also performs better for problem \eqref{e2} than the SPG algorithm in the sense that the SAPG algorithm needs less iterations and CPU time. From the comparisons between the SAPG algorithm and SPG algorithm in Figure \ref{fig3} and Figure \ref{fig4}, we can also observe that
	the SAPG algorithm significantly outperforms the SPG algorithm in terms of convergence rate when solving problem \eqref{e2} with different dimensions and sparsity levels.
	
	\begin{table}[htbp]
		{\footnotesize
			\caption{The average computational cost for Example 4.2 with different $m$, $n$ and $\mbox{Spar}$.}\label{tab2}
			\begin{center}
				\begin{tabular}{|c|c|c|c|c|c|c|c|c|c|} \hline
					
					\multicolumn{2}{|c|}{Methods} & SAPG & SPG& SAPG & SPG & SAPG & SPG & SAPG & SPG \\   \hline
					
					\multicolumn{2}{|c|}{\diagbox{Spar}{Costs}{$(m,n)$}}	& \multicolumn{2}{c|}{$(1000,200)$} &  \multicolumn{2}{c|}{$(2000,400)$} & \multicolumn{2}{c|}{$(4000,800)$}& \multicolumn{2}{c|}{$(8000,1600)$}\\ \hline
					
					\multirow{2}{*}{$20 \%$}& Iter & $\mathbf{223}$ & 250 & $\mathbf{223}$ & 269 & $\mathbf{223}$& 248& $\mathbf{223}$& 289 \\ %\hline
					\cline{2-10}
					&   Time  &  $\mathbf{0.0599}$ & 0.0674 & $\mathbf{0.1268}$ & 0.1558 & $\mathbf{0.2608}$& 0.3052& $\mathbf{1.7238}$& 2.3139 \\ \hline
					
					\multirow{2}{*}{$30 \%$}    & Iter & $\mathbf{223}$ & 434 & $\mathbf{223}$ & 433 & $\mathbf{223}$& 451& $\mathbf{223}$& 576 \\ %\hline
					\cline{2-10}
					&Time  & $\mathbf{0.0607}$ & 0.1150 & $\mathbf{0.1266}$ & 0.2497 & $\mathbf{0.2617}$& 0.5598 & $\mathbf{1.7295}$& 4.6168\\ \hline
					
					\multirow{2}{*}{$40 \%$}    & Iter & $\mathbf{223}$ & 502 & $\mathbf{223}$ & 787 & $\mathbf{223}$& 917& $\mathbf{223}$& 1162 \\ %\hline
					\cline{2-10}
					& Time & $\mathbf{0.0591}$ & 0.1297 & $\mathbf{0.1301}$ & 0.4542 & $\mathbf{0.2752}$& 1.2490& $\mathbf{1.7386}$& 9.4325 \\ \hline
					
					\multirow{2}{*}{$50 \%$}    & Iter & $\mathbf{223}$ & 1034 & $\mathbf{223}$ & 1236 & $\mathbf{223}$& 1819 & $\mathbf{223}$& 2327\\ %\hline
					\cline{2-10}
					& Time & $\mathbf{0.0615}$ & 0.2661 & $\mathbf{0.1246}$ & 0.6788 & $\mathbf{0.2640}$& 2.4045 & $\mathbf{1.7240}$& 18.4405\\ \hline
				\end{tabular}
			\end{center}
		}
	\end{table}
	
	\begin{figure}[htbp]
		\centerline{
			\subfigure[$(m,n,\mbox{Spar})=(2000,400,30\%)$.]{\label{fig1.a}\includegraphics[width=0.55\textwidth]{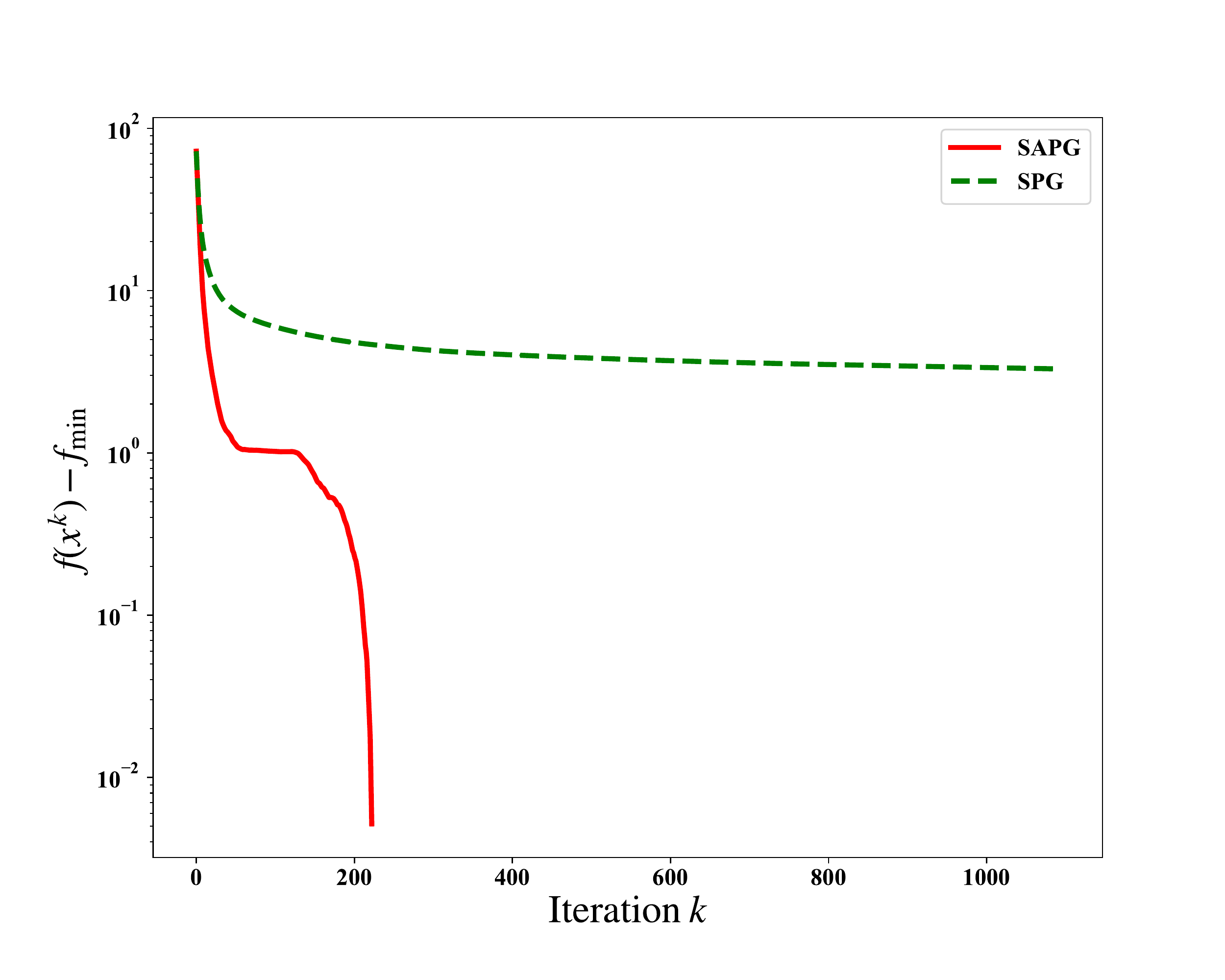}}
			\subfigure[$(m,n,\mbox{Spar})=(2000,400,50\%)$.]{\label{fig1.b}\includegraphics[width=0.55\textwidth]{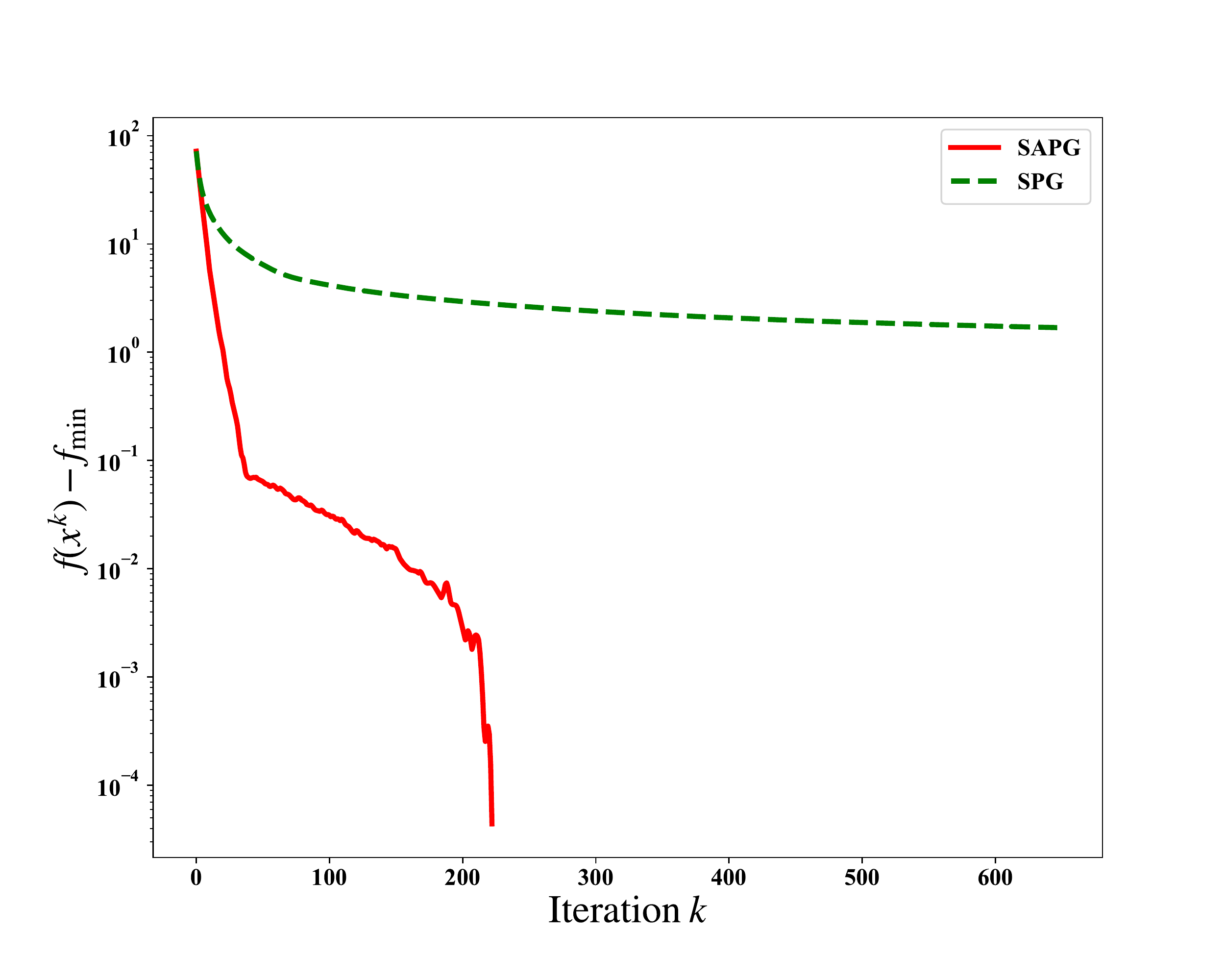}}
		}
		\caption{Convergence of $\{f(x^k)-f_{\min}\}$ for Example \ref{exa2} with $m=2000$, $n=400$ and different sparsity levels.}
		\label{fig3}       % Give a unique label
	\end{figure}
	
	\begin{figure}[htbp]
		\centerline{
			\subfigure[$(m,n,\mbox{Spar})=(8000,1600,30\%)$.]{\label{fig1.a}\includegraphics[width=0.55\textwidth]{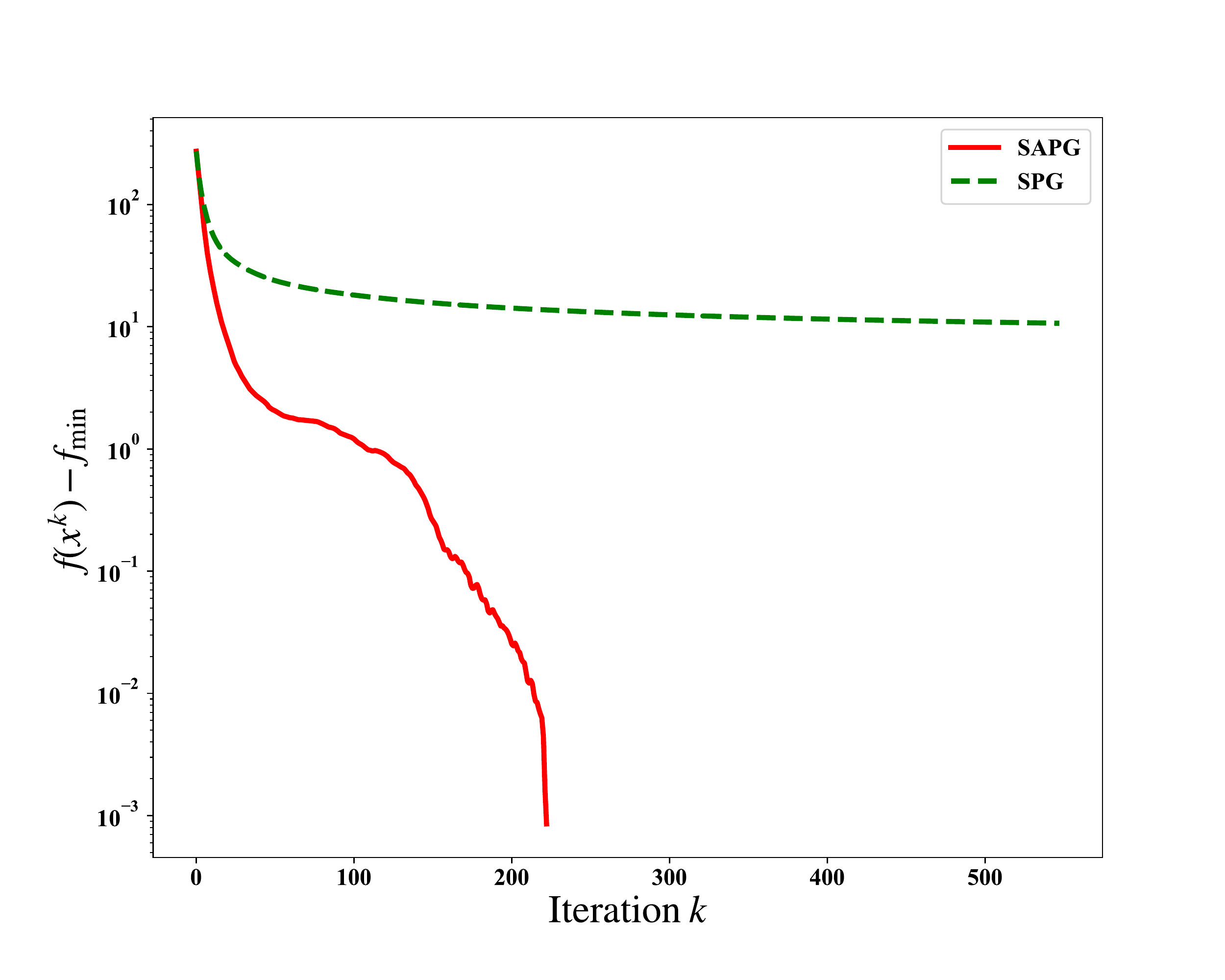}}
			\subfigure[$(m,n,\mbox{Spar})=(8000,1600,50\%)$.]{\label{fig1.b}\includegraphics[width=0.55\textwidth]{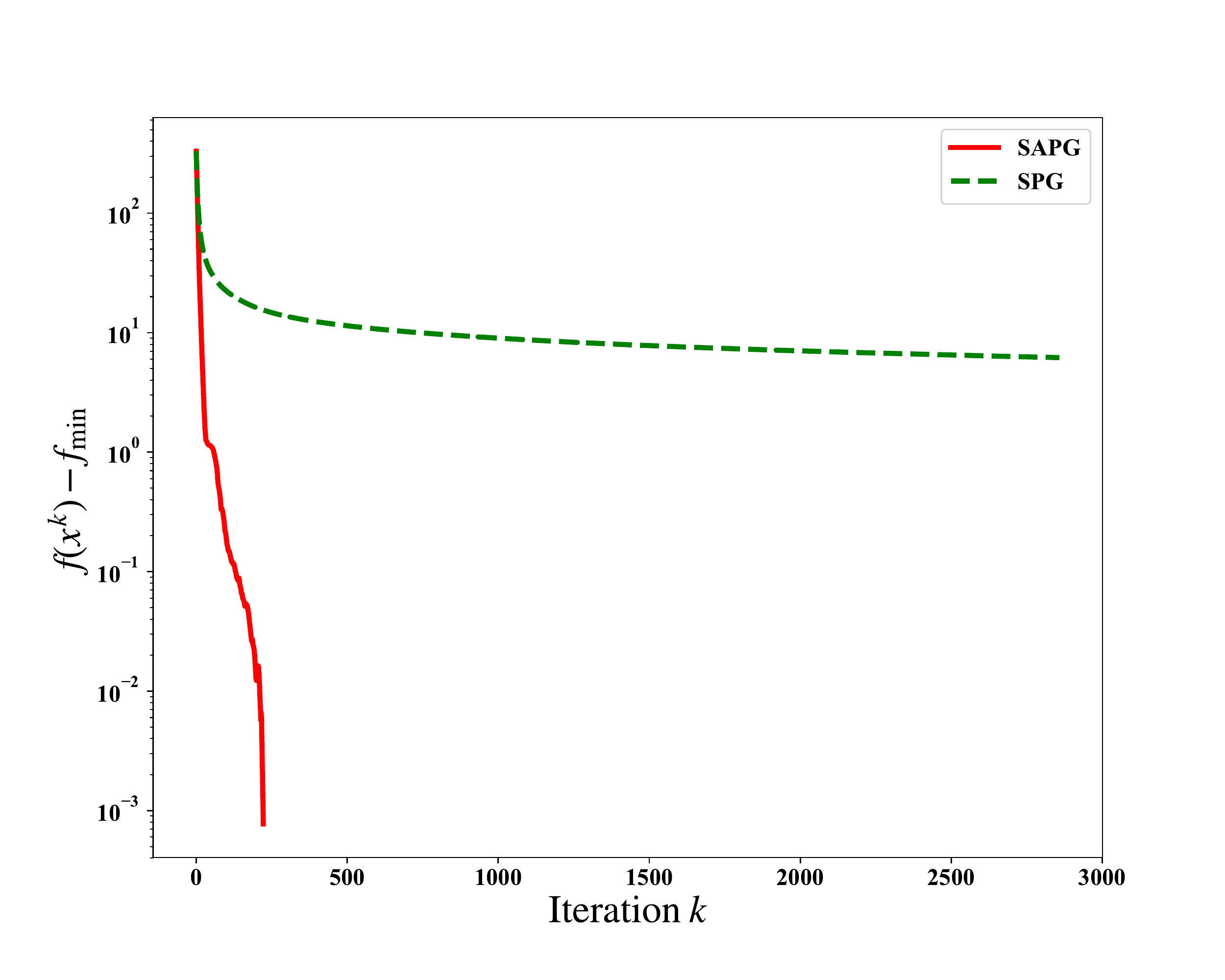}}
		}
		\caption{Convergence of $\{f(x^k)-f_{\min}\}$ for Example \ref{exa2} with $m=8000$, $n=1600$ and different sparsity levels.}
		\label{fig4}       % Give a unique label
	\end{figure}
\end{example}

From the numerical results in Example \ref{exa1} and Example \ref{exa2}, besides the faster convergence rate of the SAPG algorithm than the SPG algorithm, we have the following two observations.
\begin{itemize}
	\item [\rm (i)] From Tables \ref{tab1}-\ref{tab2}, we can see that the iteration numbers and CPU time of the SAPG algorithm are stable for all cases, while they are increasing as the sparsity is increasing for the SPG algorithm.
	This indicates that the superiority of the SAPG algorithm is highlighted when the sparsity level is large. It is surprising that the iteration number is 223 for all cases. We would like explain that the reason is that the value of $\|x^k-P_{\mathcal{X}}(x^k-\zeta\nabla\tilde{f}(x^k,\mu_k))\|_{\infty}$ generated by the SAPG algorithm decreases rapidly, and the main work of the latter iterations is to update the smoothing parameter $\mu_k$ such that $\mu_k\leq\epsilon$ and the updating method for $\mu_k$ is same for all cases of Example \ref{exa1} and Example \ref{exa2}.
	
	\item [\rm (ii)] From Figures \ref{fig1}-\ref{fig4}, the objective function values of the iterate obtained by SAPG algorithm are much smaller than those by the SPG algorithm, which shows that the SAPG algorithm can find a better $\epsilon$ minimizer with fewer iterations.
\end{itemize}

\section{Conclusions}
In this paper we develop a novel efficient smoothing accelerated proximal gradient (SAPG) algorithm for solving the constrained nonsmooth convex optimization problem modeled by \eqref{ob}, where the objective function is the sum of a continuous convex function (not necessarily smooth) and a proper closed convex function. The update method of the smoothing parameter is the essential for the convergence properties of the proposed algorithm. At each iteration, we employ the accelerated proximal gradient with extrapolation coefficient $\frac{k-1}{k+\alpha-1}$ to minimize problem \eqref{sob} with a fixed smoothing parameter. We prove that the global convergence rate of the objective function values is $o(\ln^{\sigma}k/k)$ with any $\sigma\in(\frac{1}{2}, 1]$. In addition, we show that sequence $\{x^k\}$ converges to an optimal solution of problem \eqref{ob}. Further, we propose an inexact smoothing accelerated proximal gradient (ISAPG) algorithm by introducing an error or perturbation term in the SAPG algorithm. We obtain
the fact that the convergence results of the ISAPG algorithm with appropriate perturbations are parallel to that of the SAPG algorithm.

%\paragraph{Paragraph headings} Use paragraph headings as needed.
%\begin{equation}
%\end{equation}

%\begin{figure}
%  \includegraphics{example.eps}
%\end{figure}
%\begin{figure*}
%  \includegraphics[width=0.75\textwidth]{example.eps}
%\end{figure*}

%\begin{table}

%\begin{tabular}{lll}
%\hline\noalign{\smallskip}
%\end{tabular}
%\end{table}

\begin{acknowledgements}
This work is funded by the National Science Foundation of China (Nos: 11871178).
\end{acknowledgements}


\bibliographystyle{spmpsci}      % mathematics and physical sciences
\bibliography{references}

%\begin{thebibliography}{}
% etc
%\end{thebibliography}

\end{document}